\documentclass[graybox]{svmult}

 \usepackage{amsfonts,amsmath,amstext,latexsym,amssymb,physics}
\usepackage{mathptmx}       
\usepackage{helvet}         
\usepackage{courier}        
\usepackage{type1cm}        

\usepackage{newtxtext}       %
\usepackage{newtxmath}       
\usepackage{graphicx}        
\usepackage{multicol}        
\usepackage[bottom]{footmisc}

\usepackage{cite}

\usepackage{bm}

\usepackage{enumitem}  
\usepackage{calc}
\setlist{labelindent=1pt,itemsep=0.1cm}
\setlist[itemize]{leftmargin=0.7cm}
\setlist[enumerate]{itemindent=0em,leftmargin=0.7cm}

\usepackage{float} 

\usepackage[bookmarks=false,hidelinks,unicode]{hyperref}

\DeclareMathOperator{\esssup}{ess\;sup}

\smartqed

\allowdisplaybreaks

\makeindex             

\begin{document}
\title*{Volterra integral operators and general linear integral operators representing polynomial covariance type commutation relations on $L_p$ spaces}
\titlerunning{Integral operators representing commutation relations }        

\author{Domingos Djinja \and Sergei Silvestrov \and Alex Behakanira Tumwesigye}
\authorrunning{D. Djinja, S. Silvestrov, A. B. Tumwesigye} 

\institute{Domingos Djinja \at
Department of Mathematics and Informatics, Faculty of Sciences, Eduardo Mondlane University, Box 257, Maputo, Mozambique \\
\email{domingos.djindja@uem.ac.mz}
\at
Division of Mathematics and Physics, School of Education, Culture and Communication, M{\"a}lardalen University, Box 883, 72123 V{\"a}ster{\aa}s, Sweden \\
\email{domingos.celso.djinja@mdu.se}
\and
Sergei Silvestrov
\at Division of Mathematics and Physics, School of Education, Culture and Commu\-nication, M{\"a}lardalen University, Box 883, 72123 V{\"a}ster{\aa}s, Sweden. \\
\email{sergei.silvestrov@mdu.se}
\and
Alex Behakanira Tumwesigye \at
Department of Mathematics, College of Natural Sciences, Makerere University, Box 7062, Kampala, Uganda. \email{alex.tumwesigye@mak.ac.ug}
}



\maketitle
\label{chap:DjinjaTumwSilvestVolterraintopscomrel}

\abstract*{
Conditions for linear integral operators on $L_p$ over measure spaces to satisfy the polynomial covariance type commutation relations are described in terms of defining kernels of the corresponding integral operators. Representation by integral operators are studied both for general polynomial covariance commutation relations and for important classes of polynomial covariance commutation relations associated to arbitrary monomials and to affine functions.
Representations of the covariance type commutation relations by integral Volterra operators and integral convolution operators are investigated.
}

\abstract{
Conditions for linear integral operators on $L_p$ over measure spaces to satisfy the polynomial covariance type commutation relations are described in terms of defining kernels of the corresponding integral operators. Representation by integral operators are studied both for general polynomial covariance commutation relations and for important classes of polynomial covariance commutation relations associated to arbitrary monomials and to affine functions.
Representations of the covariance type commutation relations by integral Volterra operators and integral convolution operators are investigated.
}

\section{Introduction}
Commutation relations \index{commutation relation} of the form
\begin{equation} \label{Covrelation}
  AB=B F(A)
\end{equation}
where $A, B$ are elements of an associative algebra and  $F$ is a function of the elements of the algebra, are important in many areas of Mathematics and its applications. Such commutation relations are usually called covariance relations, crossed product relations or semi-direct product relations. Elements of an algebra that satisfy \eqref{Covrelation} are called a representation of this relation.  Representations of covariance commutation relations \eqref{Covrelation} by linear operators are important for the study of actions and induced representations of groups and semigroups, crossed product operator algebras, dynamical systems, harmonic analysis, wavelets and fractals analysis and, hence have applications in physics and engineering
\cite{BratJorgIFSAMSmemo99,BratJorgbook,JorgWavSignFracbook,JorgOpRepTh88,JorMoore84,MACbook1,MACbook2,MACbook3,OstSambook,Pedbook79,Samoilenkobook}.

A description of the structure of representations for the relation \eqref{Covrelation} and more general families of self-adjoint operators satisfying such relations by bounded and unbounded self-adjoint linear operators on a Hilbert space using reordering formulas for functions of the algebra elements and operators satisfying covariance commutation relation, functional calculus and spectral representation of operators and interplay with dynamical systems generated by iteration of maps involved in the commutation relations have been considered in  \cite{BratEvansJorg2000,CarlsenSilvExpoMath07,CarlsenSilvAAM09,CarlsenSilvProcEAS10,DutkayJorg3,DJS12JFASilv,DLS09,DutSilvProcAMS,DutSilvSV,JSvT12a,JSvT12b,Mansour16,JMusondaPhdth18,JMusonda19,Musonda20,Nazaikinskii96,OstSambook,PerssonSilvestrov031,PerssonSilvestrov032,PersSilv:CommutRelDinSyst,RST16,RSST16,Samoilenkobook,SaV8894,SilPhD95,STomdynsystype1,SilWallin96,SvSJ07a,SvSJ07b,SvSJ07c,SvT09,Tomiyama87,Tomiama:SeoulLN1992,Tomiama:SeoulLN2part2000,AlexThesis2018,VaislebSa90}.

Constructions of representations of polynomial covariance commutations relations by pairs of  linear integral operators with general kernels, linear integral operators with separable kernels and linear multiplication operator  defined on some Banach spaces have been considered in \cite{DjinjaEtAll_IntOpOverMeasureSpaces,DjinjaEtAll_LinItOpInGenSepKern,DjinjaEtAll_LinMultIntOp}.

Various properties of linear Volterra integral operators have been considered in \cite{BarnesBruceVolterraOperator,EvesonNormOfVolterraOperator,GohbergKrein197067bookVolterraOpsHilbSp,KarshawOperatorNormsOfPowerVoltOp,LefevreVolterraOperatorStrictlySingular}.

In this article, we construct representations of the covariance commutation relations \eqref{Covrelation} by linear integral operators on Banach spaces $L_p$ over measure spaces, in particular, linear Volterra integral operators in
$L_p$ over the standard Lebesgue measure space in a closed interval on the real line. We derive conditions on such kernel functions so that the corresponding operators satisfy \eqref{Covrelation} for polynomial $F$ when both operators  are of linear integral type. Representation by integral operators are studied both for general polynomial covariance commutation relations and for important classes of polynomial covariance commutation relations associated to arbitrary monomials and to affine functions.
We review and improve formulations for conditions obtained in \cite{DjinjaEtAll_IntOpOverMeasureSpaces} for integral operators defined by general kernels and by convolution type kernels to satisfy the covariance type commutation relations. For convolution type kernels, these conditions are stated in terms of the zero devisors of convolution multiplication using some results from \cite{Doss1990ADS,TitchmarshDSA1929}.
Furthermore, we study the representations of commutation relations by Volterra integral operators.
We derive necessary conditions for linear Volterra type integral operators to satisfy the commutation relation \eqref{Covrelation} for  separable kernel with simple functions. We also derive sufficient conditions for linear Volterra type integral operators with separable kernels to satisfy the general commutation relation \eqref{Covrelation}.
We also derive necessary and sufficient conditions for linear Volterra type integral operators with kernels in general form to satisfy
the quantum plane relations and commutativity relations.

This paper is organized in three sections. After the introduction,  we present in Section \ref{SecPreNot} preliminaries, notations, basic definitions and two useful lemmas.
In Section \ref{SecRepreBothLI}, we present representations when both operators $A$ and $B$ are linear integral operators  acting on the Banach spaces $L_p$. In particular, we consider cases when operators are convolution type. Finally, in Section \ref{SecVoltOpRepre}  we construct representations for commutation relation \eqref{Covrelation} by linear Volterra integral operators.

\section{Preliminaries and notations}\label{SecPreNot}
In this section we present preliminaries, basic definitions and notations for this article. For further details, please read \cite{AdamsG,AkkLnearOperators,BrezisFASobolevSpaces,FollandRA,Kantarovitch,Kolmogorov,KolmogorovVol2,RudinRCA}.

Let $\mathbb{R}$ be the set of all real numbers,  $ X$ be a non-empty set. Let $(X,\Sigma, \mu)$ be a $\sigma$-finite measure space, where $\Sigma$ is a $\sigma-$algebra with measurable subsets of $X$,
and $X$ can be covered with at most countable many disjoint sets $E_1,E_2,E_3,\ldots$ such that $ E_i\in \Sigma, \,
\mu(E_i)<\infty$, $i=1,2,\ldots$  and $\mu$ is a measure. A function $f:X\to \mathbb{R}$ defined in a measure space $(X,\Sigma,\mu)$ is said to be simple if it is $\mu$-measurable and takes no more than a countable set of values.
For $1\leq p<\infty,$ we denote by $L_p(X,\mu)$, the set of all classes of equivalent (different on a set of zero measure)
measurable functions $f:X\to \mathbb{R}$ such that
$\int\limits_{X} |f(t)|^p d\mu < \infty.$
This is a Banach space  with norm
$\| f\|_p= \left( \int\limits_{X} |f(t)|^p d\mu \right)^{\frac{1}{p}}.$
We denote by $L_\infty(X,\mu)$ the set of all classes of equivalent measurable functions $f:X\to \mathbb{R}$ such that exists $\lambda>0$,
$|f(t)|\le \lambda$
almost everywhere. This is a Banach space with norm
$\displaystyle \|f\|_{\infty}=\mathop{\esssup}_{t\in X} |f(t)|.$
The support of a function $f:\, X\to\mathbb{R}$ is  ${\rm supp }\, f = \{t\in X \colon \, f(t)\not=0\}.$
We will use notation
\begin{equation*} 
  Q_{\Lambda}(u,v)=\int\limits_{\Lambda} u(t)v(t)d\mu
\end{equation*}
for $\Lambda\in \Sigma$ and such functions $u,v:\, \Lambda\to \mathbb{R} $ that integral exists and is finite.
The convolution of functions $f:\, \mathbb{R}\to\mathbb{R}$ and $g:\, \mathbb{R}\to\mathbb{R}$ is defined by \\
$
(f\star g)(t)=\int\limits_{-\infty}^{+\infty} f(t-\tau)g(\tau)d\tau.
$

\begin{theorem}[Young \cite{BrezisFASobolevSpaces}]\label{YoungThmConvolution}
  Let $(\mathbb{R},\Sigma,\mu)$ be the standard Lebesgue measure space in $\mathbb{R}$.
  Let $f\in L_1(\mathbb{R},\mu)$ and let $g\in L_p(\mathbb{R},\mu)$, with $1\leq p \leq \infty$. Then
  for almost every $t\in \mathbb{R}$ the function $s\mapsto f(t-s)g(s) $ is integrable in $\mathbb{R}$
  and we define  $(f\star g)(t)=\int\limits_{-\infty}^{+\infty} f(t-s)g(s)ds$.
  In addition $f\star g\in L_p(\mathbb{R})$ and $\|f\star g\|_{L_p}\leq \|f\|_{L_1}\|g\|_{L_p}$.
\end{theorem}

%

Now we will consider two useful lemmas for  integral operators which will be used throughout the article. Lemma \ref{lemEqIntForAllLpFunctFinitMeasureA} is used
in the proof of Theorem \ref{thmBothIntOPKernelsGen} and Lemma \ref{lemEqIntForAllLpFunctInfinitSets} is used in the proof of Theorem \ref{thmBothIntOPKernelsGenRn}. Lemma \ref{lemEqIntForAllLpFunctFinitMeasureA} is a generalization of a lemma in \cite{DjinjaEtAll_LinMultIntOp} which consider $X=[\alpha,\beta]$, $\alpha<\beta$, $\alpha,\beta\in\mathbb{R}$.

\begin{lemma}\label{lemEqIntForAllLpFunctFinitMeasureA}
Let $(X,\Sigma,\mu)$ be a $\sigma$-finite measure space. Let $G_1,G_2\in \Sigma$, $\mu(G_1)<\infty$, $\mu(G_2)<\infty$ and $f:G_1\to \mathbb{R}$, $g:G_2\to \mathbb{R}$ be two measurable functions such that for some $1\leq p \leq \infty$ and all $x\in L_p(X,\mu)$, the integrals
\begin{equation*}
  \int\limits_{G_1} f(t)x(t)d\mu,\quad \int\limits_{G_2} g(t)x(t)d\mu,
\end{equation*}
exist and are finite. Set
$
  G=G_1\cap G_2.
$
 Then the following statements are equivalent\textup{:}
\begin{enumerate}[label=\textup{\arabic*.}, ref=\arabic*]
  \item \label{LemmaFiniteMeasureEqLp:cond1} For some $1\le p \le \infty$ and all $x\in L_p(X,\mu)$,
\begin{equation*} 
 Q_{G_1}(f,x)= \int\limits_{G_1} f(t)x(t)d\mu=\int\limits_{G_2} g(t)x(t)d\mu=Q_{G_2}(g,x).
\end{equation*}
\item \label{LemmaFiniteMeasureEqLp:cond2} The following conditions hold\textup{:}
    \begin{enumerate}[label=\textup{\alph*)}]
      \item for almost every $t\in G$, $f(t)=g(t)$,
      \item for almost every $t \in G_1\setminus G,\ f(t)=0,$
      \item  for almost every $t \in G_2\setminus G,\ g(t)=0.$
    \end{enumerate}
\end{enumerate}
\end{lemma}

\begin{proof}
\noindent \ref{LemmaFiniteMeasureEqLp:cond2}$\Rightarrow$\ref{LemmaFiniteMeasureEqLp:cond1} \
By additivity of the measure of integration $\mu$ on $\Sigma$,
\begin{multline*}
\textstyle \int\limits_{G_1} f(t)x(t)d\mu = \int\limits_{G_1\setminus G} f(t)x(t)d\mu + \int\limits_{G} f(t)x(t)d\mu = \int\limits_{G} f(t)x(t)d\mu \\
\textstyle = \int\limits_{G} g(t)x(t)d\mu =  \int\limits_{G_2\setminus G} g(t)x(t)d\mu + \int\limits_{G} g(t)x(t)d\mu  =   \int\limits_{G_2} g(t)x(t)d\mu.
\end{multline*}
\noindent \ref{LemmaFiniteMeasureEqLp:cond1}$\Rightarrow$\ref{LemmaFiniteMeasureEqLp:cond2}\  For
     the indicator function $x(t)=I_{H_1}(t)$ of the set $H_1=G_1\cup G_2$,
\[ \textstyle \int\limits_{G_1} f(t)x(t)d\mu=\int\limits_{G_2} g(t)x(t)d\mu=
\int\limits_{G_1} f (t)d\mu=\int\limits_{G_2} g (t)d\mu =\eta,
\]
where $\eta$ is a constant. Now by taking $x(t)=I_{G_1\setminus G}$ we get
\[ \textstyle
\int\limits_{G_1} f(t)x(t)d\mu=\int\limits_{G_2} g(t)x(t)d\mu= \int\limits_{G_1\setminus G} f (t)d\mu=\int\limits_{G_2} g (t)\cdot 0d\mu =0.
\]
Then
    $
      \int\limits_{G_1\setminus G} f (t)d\mu =0.
    $
    Analogously by taking $x(t)=I_{G_2\setminus G}(t)$ we get
    $
      \int\limits_{G_2\setminus G} g (t)d\mu=0.
    $
    We claim that $f(t)=0$ for almost every $t\in G_1\setminus G$ and
    $g(t)=0$ for almost every $t\in G_2\setminus G$. We take an arbitrary partition of the set  $G_1\setminus G=\bigcup S_i$
    such that  $S_i\cap S_j=\emptyset$, for $i\not=j$ and each set $S_i$ has positive measure. For each $x(t)=I_{S_i}(t)$ we have
    $
    \int\limits_{G_1} f(t)x(t)d\mu=\int\limits_{G_2} g(t)x(t)d\mu=
    \int\limits_{S_i} f (t)d\mu=\int\limits_{G_2} g (t)\cdot 0d\mu =0.
    $
    Thus, for each $S_i$,
    $
          \int\limits_{S_i} f (t)d\mu=0.
    $
    Since we can choose arbitrary partition with positive measure on each of its elements,
    $
      f(t)=0 \ \mbox{ for almost every } t\in G_1\setminus G.
    $
    Analogously,
    $
      g(t)=0 \  \mbox{ for almost every } t\in G_2\setminus G.
    $
    Therefore
    $
      \eta = \int\limits_{G_1} f (t)d\mu=\int\limits_{G_2} g (t)d\mu =\int\limits_G f (t)d\mu=\int\limits_G g (t)d\mu.
    $
    Then, for all function $x\in L_p(X,\mu)$ we have
    $
      \int\limits_G f (t)x(t)d\mu=\int\limits_G g (t)x(t)d\mu \Leftrightarrow   \int\limits_G [f(t)-g(t)]x(t)d\mu=0.
    $
    By taking $x(t)=\left\{\begin{array}{cc} 1, & \mbox{if } f(t)-g(t)>0, \\ -1,  & \mbox{ if } f(t)-g(t)<0,\,   \end{array}\right.$
    almost every $t\in G$ and $x(t)=0$ for almost every $t\in X\setminus G$, we get
    $\int_{G} |f(t)-g(t)|dt =0$.
    This implies that $f(t)=g(t)$ for almost every $t\in G$.
    \qed
  \end{proof}



The following statement is similar to Lemma \ref{lemEqIntForAllLpFunctFinitMeasureA} under conditions that  $X=\mathbb{R}^l$, $f,g\in L_q(\mathbb{R}^l,\mu)$, $1<q\leq\infty$ and  the sets $G_1$, $G_2$ that can have infinite measure.
It was proved in \cite{DjinjaEtAll_LinItOpInGenSepKern}. 
\begin{lemma}[\hspace{-0.1mm}\cite{DjinjaEtAll_LinItOpInGenSepKern}]\label{lemEqIntForAllLpFunctInfinitSets}
Let  $(\mathbb{R}^l,\Sigma, \mu)$ be the standard Lebesgue measure space, $1 < q\leq \infty$, $f, g\in L_q(\mathbb{R}^l,\mu)$,
 $G_1,\, G_2\in \Sigma$
and
$
  G=G_1\cap G_2.
$
 Then the following statements are equivalent\textup{:}
\begin{enumerate}[label=\textup{\arabic*.}, ref=\arabic*]
  \item \label{LemmaAllowInfSetsEqLp:cond1} For $1 \leq p < \infty$ such that $\dfrac{1}{p}+\dfrac{1}{q}=1$ and all $x\in L_p(\mathbb{R}^l,\mu)$,
\begin{gather*}
  Q_{G_1}(f,x)=\int\limits_{G_1} f(t)x(t)d\mu=\int\limits_{G_2} g(t)x(t)d\mu=Q_{G_2}(g,x);
\end{gather*}
\item \label{LemmaAllowInfSetsEqLp:cond2} The following conditions hold\textup{:}
    \begin{enumerate}[label=\textup{\alph*)}]
      \item for almost every $t\in G$, $f(t)=g(t)$;
      \item for almost every $t \in G_1\setminus G$,
$ f(t)=0, $
 \item  for almost every $t \in G_2\setminus G$,
$
  g(t)=0.
$
    \end{enumerate}
\end{enumerate}
\end{lemma}

\section{Representations by linear integral operators}\label{SecRepreBothLI}

In this section, we construct representations of the polynomial covariance commutation relations \eqref{Covrelation} by bounded linear integral operators on Banach spaces $L_p$ over measure spaces. When $B=0$, the relation \eqref{Covrelation} is trivially satisfied for any $A$. If $A=0$ then the relation \eqref{Covrelation} reduces to $F(0)B=0$. This implies either ($F(0)=0$ and $B$ can be any well defined operator) or $B=0$. Thus, we focus on construction and properties of non-zero representations of \eqref{Covrelation}.
%

Let $(X,\Sigma,\mu)$ be a  $\sigma$-finite measure space. In this section we consider  representations of the covariance type commutation relation \eqref{Covrelation} when both $A$ and $B$ are linear integral operators  acting from the Banach space $L_p(X,\mu)$ to itself for a fixed $p$ such that $1\le p\le\infty$ defined as follows:
\begin{equation*}
  (Ax)(t)= \int\limits_{G_A} k_A(t,s)x(s)d\mu_s,\quad (Bx)(t)= \int\limits_{G_B}{k}_B(t,s)x(s)d\mu_s,
\end{equation*}
almost everywhere, where the index in $\mu_s$ indicates the variable of integration, $G_A,G_B\in \Sigma$, $\mu(G_A)<\infty$, $\mu(G_B)<\infty$,  $ k_A(t,s):X\times G_A\to \mathbb{R} $, ${k}_B(t,s):\, X\times G_B\to \mathbb{R}$   are measurable functions such that operators are well defined from $L_p(X,\mu)$ to $L_p(X,\mu)$. We will not focus on such conditions here, referring on this to \cite{ConwayFunctionalAnalysis,FollandRA}.
This allows to consider broader classes of integral operators and in this way the following theorem and its corollaries generalize the corresponding theorem and corollaries in \cite{DjinjaEtAll_IntOpOverMeasureSpaces}. Moreover, here we give a more precise proof.

\begin{theorem}\label{thmBothIntOPKernelsGen}
Let $(X,\Sigma,\mu)$ be a  $\sigma$-finite measure space and $1\le p\le\infty$, and let $A:L_p(X,\mu)\to L_p(X,\mu)$ and  $B:\, L_p(X,\mu)\to L_p(X,\mu)$ be nonzero linear operators  defined, for almost every $t$, by
\begin{align*}
& (Ax)(t)= \int\limits_{G_A} k_A(t,s)x(s)d\mu_s,\quad (Bx)(t)= \int\limits_{G_B} {k}_B(t,s)x(s)d\mu_s, \\
& \hspace{1cm} \mbox{{\rm (}the subscript $s$ in $\mu_s$ indicates the variable of integration{\rm )}}
\end{align*}
where $G_A,\, G_B \in \Sigma$, $\mu(G_A)<\infty$, $\mu(G_B)<\infty$, and $ k_A(t,s):X\times G_A\to \mathbb{R} $ and ${k}_B(t,s):X\times G_B\to \mathbb{R}$  are measurable functions.
Let
$F(z)=\sum\limits_{j=0}^{n} \delta_j z^j$, where $\delta_j \in\mathbb{R}$, $j=0,\ldots,n$. Set $G=G_A\cap G_B$ and
\begin{gather*}
k_{0,A}(t,s)=k_A(t,s), \quad
 k_{m,A}(t,s)=\int\limits_{G_A} k_A(t,\tau)k_{m-1,A}(\tau,s)d{\mu_\tau},\quad m=1,\ldots,n, \\
 F_0(k_A(t,s))=0,\quad F_n(k_A(t,s))=\sum_{j=1}^{n} \delta_j k_{j-1}(t,s),\mbox{ if } n\ge 1.
\end{gather*}
Then $AB=BF(A)$
if and only if  the following conditions are fulfilled:
 \begin{enumerate}[label=\textup{\arabic*.}, ref=\arabic*]
   \item \label{thmBothIntOPKernelsGen:cond1} for almost every $(t,\tau)\in X\times G$,
   \begin{equation*}
      \int\limits_{G_A} k_A(t,s){k}_B(s,\tau)d\mu_s-\delta_0{k}_B(t,\tau)  = \int\limits_{G_B} {k}_B(t,s) F_n(k_A(s,\tau))d\mu_s;
   \end{equation*}
   \item \label{thmBothIntOPKernelsGen:cond2} for almost every $(t,\tau)\in X\times(G_B\setminus G)$;
   \begin{equation*}
     \int\limits_{G_A} k_A(t,s){k}_B(s,\tau)d\mu_s=\delta_0 {k}_B(t,\tau),
   \end{equation*}
   \item \label{thmBothIntOPKernelsGen:cond3} for almost every $(t,\tau)\in X\times(G_A\setminus G)$,
   \begin{equation*}
     \int\limits_{G_B} {k}_B(t,s) F_n(k_A(s,\tau))d\mu_s=0.
   \end{equation*}
 \end{enumerate}
\end{theorem}

\begin{proof}
  By applying Fubini theorem from \cite{AdamsG} and iterative kernels from  \cite{KrasnolskZabreyko} we have
for all $m\ge 1$,
  \begin{gather*} 
  (A^mx)(t)=\int\limits_{G_A} k_{{m-1},A}(t,s)x(s)d\mu_s,\quad \mbox{ where }\\  \nonumber
  k_{m,A}(t,s)=\int\limits_{G_A} k_A(t,\tau)k_{m-1,A}(\tau,s)d{\mu_\tau},\quad m=1,\ldots,n,\quad  
  k_{0,A}(t,s)=k_A(t,s).
\end{gather*}
 It follows that
 \begin{eqnarray*}
   (F(A)x)(t)&=&\delta_0 x(t)+ \sum\limits_{j=1}^{n} \delta_j (A^j x)(t)\\
   &=&\delta_0 x(t)+\sum\limits_{j=1}^{n} \delta_{j} \int\limits_{G_A} k_{j-1,A}(t,s)x(s)d\mu_s \\
   &=&\delta_0 x(t)+\int\limits_{G_A} F_n(k_A(t,s))x(s)d\mu_s,\ \mbox{ where } \\
 F_0(k_A(t,s))=0,&&\quad F_n(k_A(t,s))=\sum_{j=1}^{n} \delta_j k_{j-1,A}(t,s),\mbox{ if } n\ge 1.
   \end{eqnarray*}
Computation of $BF(A)x$ and $(AB)x$ yields
  \begin{eqnarray*}
     (BF(A)x)(t)&=&\int\limits_{G_B} {k}_B(t,s)(F(A)x)(s)d\mu_s \\
   &=&\int\limits_{G_B} {k}_B(t,s) \big(\delta_0 x(s)+\int\limits_{G_A}  F_n(k_A(s,\tau)x(\tau)d{\mu_\tau}) \big)d\mu_s \\
   &=& \delta_0\int\limits_{G_B} {k}_B(t,s)x(s)d\mu_s+ \int\limits_{G_A} \big(\ \int\limits_{G_B} {k}_B(t,s)F_n(k_A(s,\tau))d\mu_s\big)x(\tau) d\mu_{\tau}\\
   &=& \delta_0\int\limits_{G_B} {k}_B(t,s)x(s)d\mu_s+ \int\limits_{G_A} k_{BF(A)}(t,\tau)x(\tau)d\mu_\tau, \\
   && \quad \text{where}\quad k_{BF(A)}(t,\tau)=\int\limits_{G_B} {k}_B(t,s) F_n(k_A(s,\tau))d\mu_s, \\
   (ABx)(t)&=&\int\limits_{G_B} k_A(t,s)(Bx)(s)d\mu_s=\int\limits_{G_A} k_A(t,s)\big(\ \int\limits_{G_B} {k}_B(s,\tau)x(\tau)d\mu_\tau \big)d\mu_s\\
   &=&\int\limits_{G_B} \big(\ \int\limits_{G_A} k_A(t,s){k}_B(s,\tau)d\mu_s\big)x(\tau)d\mu_\tau =\int\limits_{G_B} k_{AB}(t,\tau) x(\tau)d\mu_\tau,\\
  && \quad \text{where}\quad k_{AB}(t,\tau)=\int\limits_{G_A} k_A(t,s){k}_B(s,\tau)d\mu_s.
  \end{eqnarray*}

 We thus have  $(ABx)(t)=(BF(A)x)(t)$ for all $x\in L_p(X,\mu)$ if and only if
 \begin{equation*}
   \int\limits_{G_B} ( k_{AB}(t,\tau)-\delta_0 {k}_B(t,\tau)) x(\tau)d\mu_\tau=\int\limits_{G_A} k_{BF(A)}(t,\tau)x(\tau)d\mu_\tau.
 \end{equation*}
 By applying Lemma \ref{lemEqIntForAllLpFunctFinitMeasureA},    
 we conclude that $AB=BF(A)$ if and only if
 the following conditions hold:
 \begin{enumerate}[label=\textup{\arabic*.}, ref=\arabic*]
   \item for almost every $(t,\tau)\in X\times G$,
   \begin{equation*}
      \int\limits_{G_A} k_A(t,s){k}_B(s,\tau)d\mu_s-\delta_0 {k}_B(t,\tau)  = \int\limits_{G_B} {k}_B(t,s) F_n(k_A(s,\tau))d\mu_s;
   \end{equation*}
   \item for almost every $(t,\tau)\in X\times(G_B\setminus G)$,
   \begin{equation*}
     \int\limits_{G_A} k_A(t,s){k}_B(s,\tau)d\mu_s=\delta_0{k}_B(t,\tau);
   \end{equation*}
   \item for almost every $(t,\tau)\in X\times(G_A\setminus G)$,
   \begin{equation*}
     \int\limits_{G_B} {k}_B(t,s) F_n(k_A(s,\tau))d\mu_s=0. \tag*{\qed}
   \end{equation*}
 \end{enumerate}
 \end{proof}

\begin{remark} (\hspace{-0,1mm}\cite{DjinjaEtAll_IntOpOverMeasureSpaces}) \label{RemOpDefInSameInterval}
{\rm In Theorem \ref{thmBothIntOPKernelsGen} when $G_A=G_B=G$ conditions \ref{thmBothIntOPKernelsGen:cond2} and \ref{thmBothIntOPKernelsGen:cond3} are taken on set of measure zero so we can ignore them. Thus, we only remain with condition \ref{thmBothIntOPKernelsGen:cond1}. When $G_A\not=G_B$ we need to check also conditions \ref{thmBothIntOPKernelsGen:cond2} and \ref{thmBothIntOPKernelsGen:cond3} outside the intersection  $G=G_A\cap G_B$. Moreover condition \ref{thmBothIntOPKernelsGen:cond3}, which is, for almost every $(t,\tau)\in X\times(G_A\setminus G)$,
   \begin{equation}\label{Condition3ThmBothIntOpGenEx}
     \int\limits_{G_B} {k}_B(t,s) F_n(k_A(s,\tau))d\mu_s=0,
   \end{equation}
   does not imply $\displaystyle B\left(\sum_{k=1}^{n} \delta_k A^k\right)=0$ because the kernel of the integral operator $\displaystyle B\left(\sum_{k=1}^{n} \delta_k A^k\right)$ has to satisfy \eqref{Condition3ThmBothIntOpGenEx} only
   on the set $X\times (G_A\setminus G)$ and not on the whole domain of definition $X\times G_{A}$ of $k_A$. On the other hand, the same kernel has to satisfy
   condition \ref{thmBothIntOPKernelsGen:cond1}, which is,
    for almost every $(t,\tau)\in X\times G$,
   \begin{equation*}
      \int\limits_{G_A} k_A(t,s){k}_B(s,\tau)d\mu_s-\delta_0{k}_B(t,\tau)  = \int\limits_{G_B} {k}_A(t,s) F_n(k_A(s,\tau))d\mu_s.
   \end{equation*}
   Note that Theorem \ref{thmBothIntOPKernelsGen} does not imply $\displaystyle\sum_{k=1}^n \delta_k A^k=0$. In fact, $\displaystyle\sum_{k=1}^n \delta_k A^k=0$ implies $B\left(\sum\limits_{k=1}^{n} \delta_k A^k\right)=0$, but as mentioned above $B\left(\sum\limits_{k=1}^{n} \delta_k A^k\right)$ can be nonzero in general even when $A$ and $B$ satisfy the commutation relation $AB=BF(A)$.
   }
\end{remark}

\begin{theorem}[\hspace{-0.1mm}\cite{HutsonPym}]\label{TheoremHutsonPymOpwellDefinedInterval}
Let $([\alpha,\beta],\Sigma,\mu)$ be the standard Lebesgue measure in the interval $[\alpha,\beta]$, $\alpha,\beta\in\mathbb{R}$ and $\alpha<\beta$. Let $(Ax)(t)=\int\limits_{\alpha}^{\beta} k_A(t,s)x(s)d\mu_s$, where $k_A(t,s):[\alpha,\beta]^2\to \mathbb{R}$ is a function, and the subscript $s$ in $\mu_s$ denotes the variable of integration. Suppose that $k_A$ is measurable and the numbers $p,q$ are such that $1\leq p \leq \infty$, $1\leq q \leq \infty$ and $\frac{1}{p}+\frac{1}{q}=1$. Set
\begin{eqnarray*}
 \|k_A\|_{L_1}  &=& \sup_{s\in [\alpha,\beta]} \int\limits_{\alpha}^{\beta} |k(t,s)|d\mu_t, \\
  \|k_A\|_{L_p}  &=& \left(\int\limits_{\alpha}^{\beta}\left( \int\limits_{\alpha}^{\beta} |k(t,s)|^qd\mu_s\right)^{\frac{p}{q}}\right)^{\frac{1}{p}},\ (1<p<\infty), \\
   \|k_A\|_{L_\infty}  &=& \sup_{t\in [\alpha,\beta]} \int\limits_{\alpha}^{\beta} |k(t,s)|d\mu_s. \\
\end{eqnarray*}
If for some $1\leq p \leq \infty$, the quantity $\|k_A\|_{L_p}$ is finite, then the linear operator $A:L_p([\alpha,\beta])\to L_p([\alpha,\beta])$ is bounded, and $\|A\|_{L_p}\leq \|k_A\|_{L_p}$.
\end{theorem}

\begin{example} {\rm(}\hspace{-0,1mm}\cite{DjinjaEtAll_IntOpOverMeasureSpaces}{\rm)} \label{ExampleTheoremIntOpRepGenKernLp}
{\rm
Let $(\mathbb{R}, \Sigma,\mu)$ be the standard Lebesgue measure space.  Consider  integral operators acting on $L_p(\mathbb{R},\mu)$ for $1<p<\infty$.
Let $A:L_p(\mathbb{R},\mu)\to L_p(\mathbb{R},\mu)$,  $B:L_p(\mathbb{R},\mu)\to L_p(\mathbb{R},\mu)$ be defined, for almost every $t$, by
\begin{equation*}
  (Ax)(t)= \int\limits_0^\pi k_A(t,s)x(s)d\mu_s,\quad (Bx)(t)= \int\limits_0^\pi {k}_B(t,s)x(s)d\mu_s,
\end{equation*}
where the subscript $s$ in $\mu_s$ indicates the variable of integration, and
\begin{align*}
    k_A(t,s)=\, &I_{[\alpha,\beta]}(t)\frac{2}{\pi}(\cos t \cos s+\sin t\sin s+\cos t\sin s)\\
    {k}_B(t,s)=\, & I_{[\alpha,\beta]}(t)\frac{2}{\pi}(\cos t \cos s+2\sin t\sin s)
\end{align*}
for  almost every $(t,s)\in \mathbb{R}\times[0,\pi]$, $\alpha,\,\beta$ are real constants such that $\alpha\le 0$, $\beta\ge \pi$ and $I_{E}(t)$ is the indicator function of the set $E$. The operators $A$ and $B$ are well defined and bounded on $L_p(\mathbb{R},\mu)$ by  Theorem \ref{TheoremHutsonPymOpwellDefinedInterval},
since kernels $K_A$ and $K_B$ have compact support in $\mathbb{R}\times [0,\pi]$, and
\begin{eqnarray*}
  \int\limits_{\mathbb{R}}\left(\int\limits_{0}^{\pi}|k_A(t,s)|^q d\mu_s\right)^{\frac{p}{q}} d\mu_t
%
&\le & \int\limits_{\alpha}^\beta {6^p}\pi dt={6^p(\beta-\alpha)}\pi  <\infty,\\
   \int\limits_{\mathbb{R}}\left(\int\limits_{0}^{\pi}|k_B(t,s)|^q d\mu_s\right)^{\frac{p}{q}} d\mu_t
&\le & \int\limits_{\alpha}^\beta {6^p}\pi dt={6^p(\beta-\alpha)}\pi  <\infty,
\end{eqnarray*}
where $q\ge 1$ such that $\displaystyle \frac{1}{p}+\frac{1}{q}=1$. 

Note that in this case, the conditions \ref{thmBothIntOPKernelsGen:cond1}, \ref{thmBothIntOPKernelsGen:cond2} and \ref{thmBothIntOPKernelsGen:cond3} of Theorem \ref{thmBothIntOPKernelsGen}
reduce just to the condition \ref{thmBothIntOPKernelsGen:cond1} because $G_A=G_B=[0,\pi]$, and so, $G=[0,\pi]$, $G_A\setminus G=G_B\setminus G=\emptyset$. 

Consider the polynomial $F(z)=z^2$. These operators satisfy $AB=BF(A)$. 
Moreover, $BA^2\not=0$ as mentioned in Remark \ref{RemOpDefInSameInterval}. 
}
\end{example}

The following corollary follows from Theorem \ref{thmBothIntOPKernelsGen} for the deformed Heisenberg-Lie commutation relations corresponding to the case of the polynomial $F$ of degree one, or in other words to the affine mapping $F$.
\begin{corollary}
Let $(X,\Sigma,\mu)$ be a  $\sigma$-finite measure space, $1\le p\le\infty$, and let $A:L_p(X,\mu)\to L_p(X,\mu)$, $B:L_p(X,\mu)\to L_p(X,\mu)$ be nonzero linear operators defined, for almost every $t$, by
\begin{align*}
  & (Ax)(t)= \int\limits_{G_A} k_A(t,s)x(s)d\mu_s,\quad (Bx)(t)= \int\limits_{G_B} {k}_B(t,s)x(s)d\mu_s,\\
  & \quad \mbox{{\rm(}the subscript $s$ in $\mu_s$ indicates variable of integration{\rm)} }
\end{align*}
where $G_A,\ G_B\in \Sigma$, $\mu(G_A)<\infty$, $\mu(G_B)<\infty$, and $k_A(t,s):X\times G_A\to \mathbb{R} $ and ${k}_B(t,s):X\times G_B\to \mathbb{R}$ are measurable functions. Let
$F(z)=\delta_0+\delta_1 z$,\ $\delta_0,\ \delta_1\in \mathbb{R}$. Set
$G=G_A\cap G_B.$
Then,
\begin{equation*}
  AB-\delta_1 BA=\delta_0 B
\end{equation*}
if and only if  the following conditions are fulfilled:
 \begin{enumerate}[label=\textup{\arabic*.}, ref=\arabic*]
   \item for almost every $(t,\tau)\in X\times G$,
   \begin{equation*}
      \int\limits_{G_A} k_A(t,s){k}_B(s,\tau)d\mu_s-\delta_0 {k}_B(t,\tau)  = \delta_1\int\limits_{G_B} {k}_B(t,s) k_A(s,\tau)d\mu_s;
   \end{equation*}
   \item for almost every $(t,\tau)\in X\times(G_B\setminus G)$,
   \begin{equation*}
     \int\limits_{G_A} k_A(t,s){k}_B(s,\tau)d\mu_s=\delta_0 {k}_B(t,\tau);
   \end{equation*}
   \item for almost every $(t,\tau)\in X\times(G_A\setminus G)$,
   \begin{equation*}
     \delta_1\int\limits_{G_B} k_B(t,s) k_A(s,\tau)d\mu_s=0.
   \end{equation*}
 \end{enumerate}
\end{corollary}
The following corollary follows from Theorem \ref{thmBothIntOPKernelsGen} for the covariance commutation relations corresponding to the arbitrary monomial $F$. Note that for monomial of degree one, one gets the quantum plane commutation relation and if moreover the coefficient is equal to one, then one gets the commutativity relation.
\begin{corollary}\label{CorIntOpRepGenKernWRelation}
Let $(X,\Sigma,\mu)$ be a  $\sigma$-finite measure space, $1\le  p\le\infty$, and let $A:L_p(X,\mu)\to L_p(X,\mu)$,  $B:L_p(X,\mu)\to L_p(X,\mu)$ be nonzero linear operators defined, for almost every $t$, by
\begin{align*}
 & (Ax)(t)= \int\limits_{G_A} k_A(t,s)x(s)d\mu_s,\quad (Bx)(t)= \int\limits_{G_B} {k}_B(t,s)x(s)d\mu_s,\\
& \quad \mbox{{\rm(}the subscript $s$ in $\mu_s$ indicates variable of integration{\rm)} }
\end{align*}
where $G_A,\ G_B\in \Sigma$, $\mu(G_A)<\infty$, $\mu(G_B)<\infty$,  and $k_A(t,s):X\times G_A\to \mathbb{R} $ and ${k}_B(t,s):X\times G_B\to \mathbb{R}$   are measurable functions. Let
$F(z)=\delta z^d$, where $\delta\not=0$ is a nonzero real number, and $d>0$ is a positive integer. Set $G=G_A\cap G_B$ and
\begin{gather*}
k_{0,A}(t,s)=k_A(t,s), \quad
   k_{m,A}(t,s)=\int\limits_{G_A} k_A(t,\tau)k_{m-1,A}(\tau,s)d\mu_\tau,\quad m=1,\ldots,d.
\end{gather*}
Then,
\begin{equation*}
  AB=\delta BA^d
\end{equation*}
if and only if  the following conditions are fulfilled:
 \begin{enumerate}[label=\textup{\arabic*.}, ref=\arabic*]
   \item for almost every $(t,\tau)\in X\times G$,
   \begin{equation*}
      \int\limits_{G_A} k_A(t,s){k}_B(s,\tau)d\mu_s = \delta\int\limits_{G_B} {k}_B(t,s) k_{d-1,A}(s,\tau)d\mu_s.
   \end{equation*}
   \item for almost every $(t,\tau)\in X\times(G_B\setminus G)$,
   \begin{equation*}
     \int\limits_{G_A} k_A(t,s){k}_B(s,\tau)d\mu_s=0.
   \end{equation*}
   \item for almost every $(t,\tau)\in X\times(G_A\setminus G)$,
   \begin{equation*}
     \int\limits_{G_B} {k}_B(t,s) k_{d-1,A}(s,\tau)d\mu_s=0.
   \end{equation*}
 \end{enumerate}
\end{corollary}

\begin{proof}
  This follows by Theorem \ref{thmBothIntOPKernelsGen}. \qed
 \end{proof}

\begin{remark}{\rm
Example \ref{ExampleTheoremIntOpRepGenKernLp} describes a specific case for Corollary \ref{CorIntOpRepGenKernWRelation} when $G_A=G_B=[0,\pi]$, $\delta=1$, $d=2$.
}\end{remark}

Consider now the case when $X=\mathbb{R}^l$ and $\mu$ is the Lebesgue mesuare. In the following theorem we allow the sets $G_A$ and $G_B$ to have infinite measure. In this case we cannot make use of Lemma \ref{lemEqIntForAllLpFunctFinitMeasureA}, and instead, we can use Lemma \ref{lemEqIntForAllLpFunctInfinitSets}. Because of that we need conditions \eqref{ConditionsRepreRn} in Theorem \ref{thmBothIntOPKernelsGenRn}. Theorem
\ref{thmBothIntOPKernelsGenRn} has been considered in \cite{DjinjaEtAll_IntOpOverMeasureSpaces}. However, here
we give more precise conditions.
\begin{theorem}\label{thmBothIntOPKernelsGenRn}
Let $(\mathbb{R}^l,\Sigma,\mu)$ be the standard Lebesgue measure space, $1\le p<\infty$, and let $A:L_p(\mathbb{R}^l,\mu)\to L_p(\mathbb{R}^l,\mu)$ and $B:L_p(\mathbb{R}^l,\mu)\to L_p(\mathbb{R}^l,\mu)$ be nonzero operators  defined as follows, for almost very $t$,
\begin{align*}
 & (Ax)(t)= \int\limits_{G_A} k_A(t,s)x(s)d\mu_s,\quad (Bx)(t)= \int\limits_{G_B} {k}_B(t,s)x(s)d\mu_s, \\
& \quad \mbox{{\rm(}the subscript $s$ in $\mu_s$ indicates variable of integration{\rm)} }
\end{align*}
where $G_A\in\Sigma$, $G_B\in\Sigma$, and $ k_A(t,s):\mathbb{R}^l\times G_A\to \mathbb{R} $ and ${k}_B(t,s):\mathbb{R}^l\times G_B\to \mathbb{R}$   are measurable functions. Let
$F(z)=\sum\limits_{j=0}^{n} \delta_j z^j$, where $\delta_j \in\mathbb{R}$ for $j=0,\ldots,n$. Set $G=G_A\cap G_B$, and
\begin{gather*}
k_{A,0}(t,s)=k_A(t,s), \quad
   k_{A,m}(t,s)=\int\limits_{G_A} k_A(t,\tau)k_{A,{m-1}}(\tau,s)d\mu_\tau,\quad m=1,\ldots,n \\
 F_0(k_A(t,s))=0,\quad F_m(k_A(t,s))=\sum_{j=1}^{m} \delta_j k_{A,{j-1}}(t,s), \quad 1\leq m\leq n.
\end{gather*}
If $1<q\leq\infty$ with $\frac{1}{p}+\frac{1}{q}=1$, and  for almost every $t\in \mathbb{R}^l$,
\begin{gather}\label{ConditionsRepreRn}
\left.\begin{array}{c}
  I_{G_B}(\tau)\int\limits_{G_A} k_A(t,s){k}_B(s,\tau)d\mu_s \in L_q(\mathbb{R}^l,\mu), \
I_{G_B}(\cdot){k}_B(t,\cdot) \in L_q(\mathbb{R}^l,\mu), \\
I_{G_A}(\tau)\int\limits_{G_B} {k}_B(t,s)F_n(k_A(s,\tau))ds\in L_q(\mathbb{R}^l,\mu),
\end{array}\right.
\end{gather}
 then $ AB=BF(A)$ if and only if  the following conditions are satisfied:
 \begin{enumerate}[label=\textup{\arabic*.}, ref=\arabic*]
   \item \label{thmBothIntOPKernelsGenRn:cond1} for almost every $(t,\tau)\in \mathbb{R}^l\times G$,
   \begin{equation*}
      \int\limits_{G_A} k_A(t,s){k}_B(s,\tau)d\mu_s-\delta_0{k}_B(t,\tau)
       = \int\limits_{G_B} {k}_B(t,s) F_n(k_A(s,\tau))d\mu_s;
   \end{equation*}
   \item \label{thmBothIntOPKernelsGenRn:cond2} for almost every $(t,\tau)\in \mathbb{R}^l\times (G_B\setminus G)$,
   \begin{equation*}
     \int\limits_{G_A} k_A(t,s){k}_B(s,\tau)d\mu_s=\delta_0{k}_B(t,\tau);
   \end{equation*}
   \item \label{thmBothIntOPKernelsGenRn:cond3} for almost every $(t,\tau)\in \mathbb{R}^l\times(G_A\setminus G)$,
   \begin{equation*}
     \int\limits_{G_B} {k}_B(t,s) F_n(k_A(s,\tau))d\mu_s=0.
   \end{equation*}
 \end{enumerate}
\end{theorem}

\begin{proof}
  By applying Fubini theorem from \cite{AdamsG} and iterative kernels from  \cite{KrasnolskZabreyko}, we have
  by induction that for all $m\ge 1$,
  \begin{gather}\nonumber
  (A^mx)(t)=\int\limits_{G_A} k_{{m-1},A}(t,s)x(s)d\mu_s,\quad \mbox{ where }\\  \nonumber
  k_{m,A}(t,s)=\int\limits_{G_A} k_A(t,\tau)k_{{m-1},A}(\tau,s)d\mu_\tau,\quad m=1,\ldots,n,\quad
  k_{0,A}(t,s)=k_A(t,s).
\end{gather}
 It follows that
 \begin{eqnarray}\nonumber
   (F(A)x)(t)&=&\delta_0 x(t)+ \sum\limits_{j=1}^{n} \delta_j (A^j x)(t)\\ \nonumber
   &=&\delta_0 x(t)+\sum\limits_{j=1}^{n} \delta_{j} \int\limits_{G_A} k_{j-1,A}(t,s)x(s)d\mu_s \\ \nonumber
   &=&\delta_0 x(t)+\int\limits_{G_A} F_n(k_A(t,s))x(s)d\mu_s,\ \mbox{ where } \\ \nonumber
  F_0(K_A(t,s))=0, && F_n(k_A(t,s))=\sum_{j=1}^{n} \delta_j k_{{j-1},A}(t,s),\mbox{ if } n\ge 1
   \end{eqnarray}
 We now compute $BF(A)x$ and $(AB)x$. We have
 \begin{eqnarray*}
 (BF(A)x)(t)&=&\int\limits_{G_B} {k}_B(t,s)(F(A)x)(s)d\mu_s \\
  &=&\int\limits_{G_B} {k}_B(t,s) \big(\delta_0 x(s)+\int\limits_{G_A}  F_n(k_A(s,\tau))x(\tau)d\mu_\tau \big)d\mu_s\\
   &=& \delta_0\int\limits_{G_2} {k}_B(t,s)x(s)d\mu_s+ \int\limits_{G_A} \big(\ \int\limits_{G_B} {k}_B(t,s)F_n(k_B(s,\tau))d\mu_s\big)x(\tau) d\mu_\tau\\
   &=& \delta_0\int\limits_{G_B} {k}_B(t,s)x(s)d\mu_s+ \int\limits_{G_A} k_{BF}(t,\tau)x(\tau)d\mu_\tau
   \end{eqnarray*}
   where
   \begin{equation*}
   k_{BF(A)}(t,\tau)=\int\limits_{G_B} {k}_B(t,s) F_n(k_A(s,\tau))d\mu_s,
  \end{equation*}
  and
  \begin{eqnarray*}
   (ABx)(t)&=&\int\limits_{G_B} k_A(t,s)(Bx)(s)d\mu_s=\int\limits_{G_A} k_A(t,s)\big(\ \int\limits_{G_B} {k}_B(s,\tau)x(\tau)d\mu_\tau \big)d\mu_s\\
   &=&\int\limits_{G_B} \big(\int\limits_{G_A} k_A(t,s){k}_B(s,\tau)d\mu_s\big)x(\tau)d\mu_\tau =
  \int\limits_{G_B} k_{AB}(t,\tau) x(\tau)d\mu_\tau,
  \end{eqnarray*}
  where
  \begin{equation*}
    k_{AB}(t,\tau)=\int\limits_{G_A} k_A(t,s){k}_B(s,\tau)d\mu_s.
 \end{equation*}
 Thus, for all $x\in L_p(\mathbb{R}^l,\mu)$,  we have  $(ABx)(t)=(BF(A)x)(t)$ for almost every $t$  if and only if
 \begin{equation*}
   \int\limits_{G_B} ( k_{AB}(t,\tau)-\delta_0 {k}_B(t,\tau)) x(\tau)d\mu_\tau=\int\limits_{G_A} k_{BF(A)}(t,\tau)x(\tau)d\mu_\tau,
 \end{equation*}
 for almost every $t$. By applying Lemma \ref{lemEqIntForAllLpFunctInfinitSets},
 we conclude that $AB=BF(A)$ if and only if
 the following conditions are satisfied:
 \begin{enumerate}[label=\textup{\arabic*.}, ref=\arabic*]
   \item for almost every $(t,\tau)\in \mathbb{R}^l\times G$,
   \begin{equation*}
      \int\limits_{G_A} k_A(t,s){k}_B(s,\tau)d\mu_s-\delta_0{k}_B(t,\tau)  = \int\limits_{G_B} {k}_B(t,s) F_n(k_A(s,\tau))d\mu_s;
   \end{equation*}
   \item for almost every $(t,\tau)\in \mathbb{R}^l\times(G_B\setminus G)$,
   \begin{equation*}
     \int\limits_{G_A} k_A(t,s){k}_B(s,\tau)d\mu_s=\delta_0{k}_B(t,\tau);
   \end{equation*}
   \item for almost every $(t,\tau)\in \mathbb{R}^l\times(G_A\setminus G)$,
   \begin{equation*}
     \int\limits_{G_B} {k}_B(t,s) F_n(k_A(s,\tau))d\mu_s=0.
   \end{equation*}
 \end{enumerate}
 \qed
\end{proof}

The  Propositions \ref{PropConvolutionGeneralPolyR}, \ref{PropConvolutionMonomialR} and \ref{PropConvolutionGenPolyOnesided}  were proved in \cite{DjinjaEtAll_IntOpOverMeasureSpaces}.
\begin{proposition}\label{PropConvolutionGeneralPolyR}
For  $1< p<\infty$ and $L_p(\mathbb{R},\mu)$ on the standard Lebesgue measure space $(\mathbb{R},\Sigma,\mu)$,
let $A:L_p(\mathbb{R},\mu)\to L_p(\mathbb{R},\mu)$ and $B:L_p(\mathbb{R},\mu)\to L_p(\mathbb{R},\mu)$ be nonzero linear operators defined, for almost every $t$, by
\begin{equation} \label{OperatorsAandBinconvolutionForm}
  (Ax)(t)= \int\limits_{\mathbb{R}} \tilde{k}_A(t-s)x(s)d\mu_s,\quad (Bx)(t)= \int\limits_{\mathbb{R}} \tilde{k}_B(t-s)x(s)d\mu_s,
\end{equation}
where $\tilde{k}_A(\cdot)\in L_1(\mathbb{R},\mu) $, $\tilde{k}_B(\cdot)\in L_1(\mathbb{R},\mu)$, that is,
 \begin{gather*} 
  \int\limits_{\mathbb{R}}|\tilde{k}_A(t)|d\mu_t <\infty,\qquad
    \int\limits_{\mathbb{R}}|\tilde{k}_B(t)|d\mu_t <\infty.
\end{gather*}
Let
$F(z)=\sum\limits_{j=1}^{n} \delta_j z^j$, where $\delta_j \in\mathbb{R}$ for $j=1,\ldots,n$.
Suppose that
\begin{gather*}
K_A(s)=\int\limits_{-\infty}^{\infty} \exp({-st})\tilde{k}_A(t)d\mu_t,\quad K_B(s)=\int\limits_{-\infty}^{\infty} \exp({-st})\tilde{k}_B(t)d\mu_t,
\end{gather*}
and the domains of $K_A(\cdot)$ and $K_B(\cdot)$, where the integrals exist and finite, are equal up to a set of measure zero.
If $k_B(\cdot)\in L_q(\mathbb{R},\mu)$,
$1<q<\infty$, $\frac{1}{p}+\frac{1}{q}=1$,
then
$
  AB=BF(A)
$
if and only if \
$
{\rm supp }\, K_B \, \cap \, {\rm supp }\, \big( K_A -\sum\limits_{j=1}^n \delta_j K_A^j\big)
$
has measure zero in $\mathbb{R}$,
\end{proposition}

\begin{proposition}\label{PropConvolutionMonomialR}
For  $1< p<\infty$ and $L_p(\mathbb{R},\mu)$ on the standard Lebesgue measure space $(\mathbb{R},\Sigma,\mu)$,
let $A:L_p(\mathbb{R},\mu)\to L_p(\mathbb{R},\mu),\  B:L_p(\mathbb{R},\mu)\to L_p(\mathbb{R},\mu)$
be nonzero linear operators defined, for almost every $t$, by
\begin{equation*} 
  (Ax)(t)= \int\limits_{\mathbb{R}} \tilde{k}_A(t-s)x(s)d\mu_s,\quad (Bx)(t)= \int\limits_{\mathbb{R}} \tilde{k}_B(t-s)x(s)d\mu_s,
\end{equation*}
where  $\tilde{k}_A(\cdot)\in L_1(\mathbb{R},\mu) $, $\tilde{k}_B(\cdot)\in L_1(\mathbb{R},\mu)$, that is,
 \begin{equation*} 
  \int\limits_{\mathbb{R}}|\tilde{k}_A(t)|d\mu_t <\infty,\quad \int\limits_{\mathbb{R}}|\tilde{k}_B(t)|d\mu_t <\infty
\end{equation*}
and the subscripts in $\mu_s$ and $\mu_t$ indicate the variable of integration.
Suppose that
\[
  \int\limits_{-\infty}^{\infty} \exp({-st})\tilde{k}_A(t)d\mu_t={K}_A(s),\quad \int\limits_{-\infty}^{\infty} \exp({-st})\tilde{k}_B(t)d\mu_t={K}_B(s)
\]
and the domains of $K_A(\cdot)$ and $K_B(\cdot)$, where the integrals exist and finite, are equal up to a set of measure zero.
If $k_B(\cdot)\in L_q(\mathbb{R},\mu)$, 
 $1<q<\infty$, $\frac{1}{p}+\frac{1}{q}=1$,
 then, 
$AB=\delta BA^n$,
for a fixed $n\in\mathbb{Z},\ n\ge 2$ and $\delta\in\mathbb{R}\setminus\{0\}$ if and only if $(\tilde{k}_A \star \tilde{k}_B)(t)=0$ almost everywhere.
\end{proposition}

\begin{remark}{\rm
Let $(\mathbb{R},\Sigma,\mu)$ be the standard Lebesgue measure space.
Consider operators $A$ and $B$ defined in \eqref{OperatorsAandBinconvolutionForm}, that is, $A:\, L_p(\mathbb{R},\mu)\to L_p(\mathbb{R},\mu)$,
$B:\, L_p(\mathbb{R},\mu)\to L_p(\mathbb{R},\mu)$, $1< p<\infty$, and for almost every $t$,
\begin{align*}
&  (Ax)(t)= \int\limits_{\mathbb{R}} \tilde{k}_A(t-s)x(s)d\mu_s,\quad (Bx)(t)= \int\limits_{\mathbb{R}} \tilde{k}_B(t-s)x(s)d\mu_s,\\
& \quad \mbox{{\rm(}the subscript $s$ in $\mu_s$ indicates variable of integration{\rm)} }
\end{align*}
with $\tilde{k}_A(\cdot)\in L_1(\mathbb{R},\mu)$, $\tilde{k}_B(\cdot)\in L_1(\mathbb{R},\mu)$. If $(\tilde{k}_A\star \tilde{k}_B)(\cdot)\in L_q(\mathbb{R},\mu)$,  $1<q<\infty$, $\frac{1}{p}+\frac{1}{q}=1$, then
  $AB=BA$. In fact, by applying Fubini theorem for composition of operators $A$, $B$ and Lemma  \ref{lemEqIntForAllLpFunctInfinitSets}
we have $AB=BA$ if and only if
\begin{eqnarray*}
\int\limits_{\mathbb{R}} \tilde{k}_A(t-s) \tilde{k}_B(s-\tau)ds &=&\int\limits_{\mathbb{R}} \tilde{k}_B(t-s) \tilde{k}_A(s-\tau)d\mu_s
 \Longleftrightarrow \\
(\tilde{k}_A\star \tilde{k}_B)(t-\tau)&=&(\tilde{k}_B\star \tilde{k}_A)(t-\tau)
\end{eqnarray*}
for almost every $(t,\tau)\in\mathbb{R}^2$. This holds by the commutativity property of convolution.
If $\tilde{k}_A(\cdot)\in L_q(\mathbb{R},\mu)$ and $\tilde{k}_B(\cdot) \in L_q(\mathbb{R},\mu)$, then the condition $(\tilde{k}_A\star \tilde{k}_B)(\cdot)\in L_q(\mathbb{R},\mu)$ can be replaced by the sufficient condition $\tilde{k}_B(\cdot)\in L_q(\mathbb{R},\mu)$. It follows by \cite[Theorem 4.15]{BrezisFASobolevSpaces} that $(\tilde{k}_A\star \tilde{k}_B)(\cdot)\in L_q(\mathbb{R},\mu)$. However, this is a sufficient condition and not necessary. So, condition $(\tilde{k}_A\star \tilde{k}_B)(\cdot)\in L_q(\mathbb{R},\mu)$ is weaker than to require that $\tilde{k}_B(\cdot)\in L_q(\mathbb{R},\mu)$.
}\end{remark}

\begin{proposition}\label{PropConvolutionGenPolyOnesided}
For  $1< p<\infty$ and $L_p([0,\infty[,\mu)$ on the standard Lebesgue measure space $([0,\infty[,\Sigma,\mu)$,
let \[A:L_p([0,\infty[,\mu)\to L_p([0,\infty[,\mu),\  B:L_p([0,\infty[,\mu)\to L_p([0,\infty[,\mu)\]
be non-zero linear operators defined, for almost every $t$, by
\begin{align*}
&
\begin{array}{lll}
(Ax)(t) &=& \int\limits_{0}^\infty \tilde{k}_A(t-s)I_{[0,\infty[}(t-s)x(s)d\mu_s, \\
(Bx)(t) &=& \int\limits_{0}^\infty \tilde{k}_B(t-s)I_{[0,\infty[}(t-s)x(s)d\mu_s,
\end{array}
\\
&\tilde{k}_A(\cdot)\in L_1([0,\infty[,\mu),\ \tilde{k}_B(\cdot)\in L_1([0,\infty[,\mu)   \\
& \text{that is,}\ \textstyle \int\limits_{0}^\infty|\tilde{k}_A(t)|d\mu_t <\infty,\ \int\limits_{0}^\infty|\tilde{k}_B(t)|d\mu_t <\infty,
\end{align*}
where $I_{E}(\cdot)$ is the indicator function of the set $E$.
If $\tilde{k}_B(\cdot)\in L_q([0,\infty[,\mu)$, 
$1<q<\infty$, $\frac{1}{p}+\frac{1}{q}=1$,
 then, there are no non-zero operators  $A$ and $B$ satisfying
$AB=\delta BA^n
$
for a fixed $n\in \mathbb{Z},\ n\ge 2$, $\delta\in\mathbb{R}\setminus\{0\}$.  
\end{proposition}

\begin{remark}
Propositions \ref{PropConvolutionGeneralPolyR},  \ref{PropConvolutionMonomialR} and  \ref{PropConvolutionGenPolyOnesided} have been considered in
\cite{DjinjaEtAll_IntOpOverMeasureSpaces}, however here we add the condition $k_B(\cdot)\in L_q(\mathbb{R},\mu)$. By Young theorem (Theorem \ref{YoungThmConvolution}), this condition implies \eqref{ConditionsRepreRn}, hence both Theorem \ref{thmBothIntOPKernelsGenRn} and Lemma \ref{lemEqIntForAllLpFunctInfinitSets} can be applied.
\end{remark}

\section{Representations of commutation relation by Volterra operators }\label{SecVoltOpRepre}

Let $([\alpha,\beta],\Sigma,\mu)$ be the standard Lebesgue measure in the interval $[\alpha,\beta]$, $\alpha,\beta\in\mathbb{R}$, $\alpha<\beta$. From now on we denote $L_p([\alpha,\beta],\mu)=L_p[\alpha,\beta]$, $1\leq p\leq \infty$ and $d\mu_s=ds$.
We consider simple functions,
\begin{equation*}
\resizebox{0.95\hsize}{!}{$\displaystyle a(t)=\sum_{k=1}^{\infty} a_k I_{D_{a_k}}(t), \, b(t)=\sum_{i=1}^{\infty} b_i I_{D_{a_i}}(t),\, c(t)=\sum_{j=1}^{\infty} c_j I_{D_{c_j}}(t),\, e(t)=\sum_{l=1}^{\infty} e_l I_{D_{e_l}}(t), $}
\end{equation*}
where $\{D_{a_k}\}$, $\{D_{b_i}\}$, $\{ D_{c_j}\}$, $\{D_{e_l}\}$ are measurable partitions of $[\alpha,\beta]$,
$a_k$, $b_i$, $c_j$, $e_l$ are constants, $I_E$ is the indicator (characteristic) function of the set $E$.
%
 Whenever we consider such simple functions, we assume that
each pair of sets in each partition is such that
$\sup D_{a_{k_1}} \leq \inf D_{a_{k_2}}$ for all $k_1,k_2$ with $k_1< k_2$,
and moreover
$$ \bigcup\limits_{i,j,k,l=1}^{\infty}\{ D_{a_k} \cap D_{b_i} \cap D_{c_j}\cap D_{e_l} \}=\bigcup\limits_{m=1}^\infty P_m $$ is a partition of $[\alpha,\beta]$ such that
each pair of set in the  partition satisfies the following:  
$\sup P_{m_{1}} \leq \inf P_{m_{2}}$ for all $m_1,m_2$ with $m_1< m_2$.

\begin{proposition}\label{propOpIntVoltTypeSimple}
Let  $1\leq p\leq\infty$, $1\leq q\leq \infty$ be such that $\frac{1}{p}+\frac{1}{q}=1$, and
$$A:\,L_p(\mathbb{[\alpha,\beta]})\to L_p(\mathbb{[\alpha,\beta]}),\quad B:\,L_p(\mathbb{[\alpha,\beta]})\to L_p(\mathbb{[\alpha,\beta]})$$
be linear operators defined, for almost every $t$, by
\begin{gather*} 
(Ax)(t)= \int\limits_{\alpha}^{t} a(t)c(s)x(s)ds,\quad (Bx)(t)= \int\limits_{\alpha}^{t} b(t)e(s)x(s)ds,
\end{gather*}
where $\alpha$ is a real number, $a, b, c, e$ are measurable simple functions,  $a, b \in L_p(\mathbb{[\alpha,\beta]})$ and $c, e\in L_q(\mathbb{[\alpha,\beta]})$. Let  $F(z)=\sum\limits_{n=0}^{\deg(F)}\delta_n z^n$, $\delta_n\in\mathbb{R}$, $n=0,\ldots,\deg(F)$, $\deg(F)\ge 2$.

If  $AB=BF(A)$, then
$
 {\rm supp}\, \{ abce\}
$
is a set of measure zero.
\end{proposition}

\begin{proof}
 Let $n\ge 1$ and we compute $AB$, $A^n$ and $BA^n$. Let $a$, $b$, $c$ and $e$ be measurable simple functions defined as follows
\begin{equation*}
\resizebox{0.95\hsize}{!}{$\displaystyle a(t)=\sum_{k=1}^{\infty} a_k I_{D_{a_k}}(t), \, b(t)=\sum_{i=1}^{\infty} b_i I_{D_{a_i}}(t),\, c(t)=\sum_{j=1}^{\infty} c_j I_{D_{c_j}}(t),\, e(t)=\sum_{l=1}^{\infty} e_l I_{D_{e_l}}(t), $}
\end{equation*}
where $\{D_{a_k}\}$, $\{D_{b_i}\}$, $\{ D_{c_j}\}$, $\{D_{e_l}\}$ are measurable partitions of $[\alpha,\beta]$,
$a_k$, $b_i$, $c_j$, $e_l$ are constants, $I_E$ is the indicator (characteristic) function of the set $E$.
We have
\begin{eqnarray*}
 && (ABx)(t)=\int\limits_{\alpha}^{t} a(t)c(s)(Bx)(s)ds=\int\limits_{\alpha}^{t} a(t)c(s_1)\big(\int\limits_{\alpha}^{s_1} b(s_1)e(s_2)x(s_2)ds_2\big)ds_1 \\
  &&=\int\limits_{\alpha}^{t} \sum_{k=1}^\infty \sum_{j=1}^\infty a_kc_j I_{D_{a_k}}(t)
  I_{D_{c_j}}(s_1)
  \big(\int\limits_{\alpha}^{s_1} \sum_{i=1}^\infty \sum_{l=1}^\infty b_ie_l I_{D_{b_i}}(s_1)
  I_{D_{e_l}}(s_2)x(s_2)ds_2\big)ds_1,
  \end{eqnarray*}
  almost everywhere. We compute $A^n$, starting with $A^2$ and $A^3$. For almost every $t$,
  \begin{align*}
  &(A^2x)(t)=\int\limits_{\alpha}^{t} a(t)c(s)(Ax)(s)ds=\int\limits_{\alpha}^{t} a(t)c(s)\big(\int\limits_{\alpha}^{s}a(s)c(\tau)x(\tau)d\tau\big)ds\\
  &=\int\limits_{\alpha}^{t} \sum_{k=1}^\infty \sum_{j=1}^\infty a_kc_j I_{D_{a_k}}(t)
  I_{D_{c_j}}(s_1)\\ 
  &\hspace{2cm} \cdot \big(\int\limits_{\alpha}^{s_1}\sum_{k=1}^\infty \sum_{j=1}^\infty a_kc_j I_{D_{a_k}}(s_1)
  I_{D_{c_j}}(s_2)x(s_2)ds_2\big)ds_1, \\
  &(A^3x)(t)=\int\limits_{\alpha}^{t} a(t)c(s)(A^2x)(s)ds
  =\hspace{0mm}\int\limits_{\alpha}^{t} a(t)c(s_1)\\
  & \hspace{1cm} \cdot \big(\int\limits_{\alpha}^{s_1}a(s_1)c(s_2)\hspace{0mm}
  \hspace{0cm} \big(\int\limits_\alpha^{s_2}a(s_2)c(s_3)
  x(s_3)ds_3\big)
 ds_2\big)ds_1
 \\
 &=
  \int\limits_{\alpha}^{t} \sum_{k=1}^\infty \sum_{j=1}^\infty a_kc_j I_{D_{a_k}}(t)
  I_{D_{c_j}}(s_1)\\
  & \hspace{1cm} \cdot \bigg(\int\limits_{\alpha}^{s_1} \sum_{k=1}^\infty \sum_{j=1}^\infty a_kc_j I_{D_{a_k}}(s_1) I_{D_{c_j}}(s_2)
  \\
   & \hspace{2cm}
  \cdot \big(\int\limits_\alpha^{s_2} \sum_{k=1}^\infty \sum_{j=1}^\infty a_kc_j I_{D_{a_k}}(s_2)
  I_{D_{c_j}}(s_3)x(s_3)ds_3\big)ds_2\bigg)ds_1.
  \end{align*}
  Therefore, for $n\ge 1$ we have
  \begin{eqnarray*}
 && (A^n x)(t)=\int\limits_{\alpha}^{t} a(t)b(s_1)(A^{n-1}x)(s_1)ds_1=\int\limits_{\alpha}^{t} a(t)b(s_1) ds_1
   \\
   & &
 \hspace{3cm}  \cdot \left(
  \prod\limits_{i=1}^{n-2} \int\limits_{\alpha}^{s_i}a(s_i)b(s_{i+1})ds_{i+1}\right)
  \int\limits_{\alpha}^{s_{n-1}}a(s_{n-1})b(s_n)x(s_n)ds_n\\
  &&
  = \int\limits_{\alpha}^{t} \sum_{k=1}^\infty \sum_{j=1}^\infty a_kc_j I_{D_{a_k}}(t)  I_{D_{c_j}}(s_1)ds_1 \\
&&\hspace{2cm} \cdot\big(\prod\limits_{m=1}^{n-2} \int\limits_{\alpha}^{s_m} \sum_{k=1}^\infty \sum_{j=1}^\infty a_kc_j I_{D_{a_k}}(s_{m})
  I_{D_{c_j}}(s_{m+1})ds_{m+1}\big)\\
  & &
  \hspace{3cm}
  \cdot
  \int\limits_{\alpha}^{s_{n-1}} \sum_{k=1}^\infty \sum_{j=1}^\infty a_kc_j I_{D_{a_k}}(s_{n-1})
  I_{D_{c_j}}(s_n)x(s_n)ds_n.
  \end{eqnarray*}
  Hence,  for $n\ge 1$ and for almost every $t$,
  \begin{eqnarray*}\hspace{-0.1cm}
 && (BA^n x)(t)=\int\limits_{\alpha}^{t} c(t)e(s_1) ds_1\big(
  \prod\limits_{m=1}^{n-1} \int\limits_{\alpha}^{s_i}a(s_i)b(s_{m+1})ds_{m+1}\big)\\
 &&
 \hspace{0cm} \cdot\int\limits_{\alpha}^{s_{n}}a(s_{n})b(s_{n+1})
   x(s_{n+1})ds_{n+1}= \int\limits_{\alpha}^{t} \sum_{i=1}^\infty \sum_{l=1}^\infty b_ie_l I_{D_{b_i}}(s_1)
  I_{D_{e_l}}(s_1)ds_1\\
  &&
 \hspace{4cm} \cdot\big(
  \prod\limits_{m=1}^{n-1} \int\limits_{\alpha}^{s_m} \sum_{k=1}^\infty \sum_{j=1}^\infty a_kc_j I_{D_{a_k}}(s_{m})
   I_{D_{c_j}}(s_{m+1})ds_{m+1}\big)\\
   &&
\hspace{3cm}   \cdot
  \int\limits_{\alpha}^{s_{n}} \sum_{k=1}^\infty \sum_{j=1}^\infty a_kc_j I_{D_{a_k}}(s_{n})
  I_{D_{c_j}}(s_{n+1})x(s_{n+1})ds_{n+1}.
\end{eqnarray*}
If $(AB)x=BF(A)x$  for all $x\in L_p[\alpha,\beta]$, $1< p<\infty$, then
\begin{eqnarray*} 
&& \hspace{-1mm} \int\limits_{\alpha}^{t} \sum_{k=1}^\infty \sum_{j=1}^\infty a_kc_j I_{D_{a_k}}(t)
  I_{D_{c_j}}(s_1)
  \big(\int\limits_{\alpha}^{s_1} \sum_{i=1}^\infty \sum_{l=1}^\infty b_ie_l I_{D_{b_i}}(s_1)
  I_{D_{e_l}}(s_1)x(s_2)ds_2\big)ds_1\\ 
&&\hspace{-1mm}  =
 \delta \int\limits_{\alpha}^{t} \sum_{i=1}^\infty \sum_{l=1}^\infty b_ie_l I_{D_{b_i}}(t)
  I_{D_{e_l}}(s_1)ds_1\hspace{0mm}\\ 
  &&
 \hspace{2cm} \cdot\big(
  \prod\limits_{m=1}^{n-1} \int\limits_{\alpha}^{s_m} \sum_{k=1}^\infty \sum_{j=1}^\infty a_kc_j I_{D_{a_k}}(s_{m})
  I_{D_{c_j}}(s_{m+1})ds_{m+1}\big)\hspace{-0.5mm} \\
 &&
 \hspace{3cm}
 \ \cdot
  \int\limits_{\alpha}^{s_{n}} \sum_{k=1}^\infty \sum_{j=1}^\infty a_kc_j I_{D_{a_k}}(s_{n})
  I_{D_{c_j}}(s_{n+1})x(s_{n+1})ds_{n+1}.  
\end{eqnarray*}

 For $x(\cdot)=1$ and almost every $ t\in  [\alpha,\beta]$,
\begin{eqnarray*}
 && (ABx)(t)=(B(F(A))x)(t)\, \Leftrightarrow
  \int\limits_{\alpha}^{t} \lim_{n_a,n_c\to\infty} \sum_{k=1}^{n_a} \sum_{j=1}^{n_c} a_kc_j I_{D_{a_k}}(t)
  I_{D_{c_j}}(s_1)
\\
&& \hspace{2cm}
\cdot\big(\int\limits_{\alpha}^{s_1} \lim_{n_b,n_e\to\infty}\sum_{i=1}^{n_b} \sum_{l=1}^{n_e} b_ie_l I_{D_{b_i}}(s_1)I_{D_{e_l}}(s_2)ds_2\big)ds_1
 \\
 &&
  =
\delta \int\limits_{\alpha}^{t} \lim_{n_b,n_e\to \infty}\sum_{i=1}^{n_b} \sum_{l=1}^{n_e} b_ie_l I_{D_{b_i}}(t)
 I_{D_{e_l}}(s_1)ds_1
 \\
 &&
\hspace{1cm}
\cdot \big(  \prod\limits_{m=1}^{n-1} \int\limits_{\alpha}^{s_m} \lim_{n_a,n_c\to\infty} \sum_{k=1}^{n_a} \sum_{j=1}^{n_c} a_kc_j I_{D_{a_k}}(s_{m})
  I_{D_{c_j}}(s_{m+1})ds_{m+1}\big) \\
&&  \hspace{2cm} \cdot
  \int\limits_{\alpha}^{s_{n}} \lim_{n_a,n_c\to\infty} \sum_{k=1}^{n_a} \sum_{j=1}^{n_c} a_kc_j I_{D_{a_k}}(s_{n})
  I_{D_{c_j}}(s_{n+1})ds_{n+1}. 
\end{eqnarray*}
Note that $ \bigcup\limits_{i,j,k,l=1}^{\infty}\{ D_{a_k} \cap D_{b_i} \cap D_{c_j}\cap D_{e_l} \}=\bigcup\limits_{m=1}^\infty P_m $ is a partition of $[\alpha,\beta]$. Therefore, for $x(\cdot)=1$, $t\in P_1=D_{a_1} \cap D_{b_1} \cap D_{c_1}\cap D_{e_1}$ and by applying the dominated convergence theorem we have
\begin{align*}
& \hspace{0mm}(ABx)(t)=\hspace{0mm}\int\limits_{\alpha}^{t}  a_1c_1 I_{D_{a_1}}(t)
  I_{D_{c_1}}(s_1) 
\hspace{0mm}  b_1e_1 (s_1I_{D_{e_1}}(s_1)-\gamma_{e_1})I_{D_{b_1}}(s_1) ds_1\\
&=
a_1c_1b_1e_1
\big(\frac{t^2}{2} I_{D_{b_1}\cap D_{e_1}\cap D_{a_1}\cap D_{c_1} }(t)
-\frac{\gamma_{c_1,b_1,e_1}^2}{2}I_{D_{e_1}\cap D_{c_1}\cap D_{b_1}}(\gamma_{c_1,b_1,e_1})I_{D_{a_1}}(t)
\\
&
- t\cdot \gamma_{e_1}I_{D_{b_1}\cap D_{e_1}\cap D_{a_1}\cap D_{c_1} }(t)
 +\gamma_{e_l}
\gamma_{c_1,b_1,e_1}I_{D_{e_1}\cap D_{c_1}\cap D_{b_1}}(\gamma_{c_1,b_1,e_1})I_{D_{a_1}}(t)\big)
\\
&
= a_1c_1b_1e_1
\big(\frac{t^2}{2} I_{D_{b_1}\cap D_{e_1}\cap D_{a_1}\cap D_{c_1} }(t)
-\frac{\gamma_{c_1,b_1,e_1}^2}{2}I_{D_{e_1}\cap D_{c_1}\cap D_{b_1}}(\gamma_{c_1,b_1,e_1})I_{D_{a_1}}(t)
\\
&
- t\cdot \gamma_{e_1}I_{D_{b_1}\cap D_{e_1}\cap D_{a_1}\cap D_{c_1} }(t)
  +\gamma_{e_1}
\gamma_{c_1,b_1,e_1}I_{D_{e_1}\cap D_{c_1}\cap D_{b_1}}(\gamma_{c_1,b_1,e_1})I_{D_{a_1}}(t)\big),
\end{align*}
where $\gamma_{e_l}=\inf D_{e_l}$, $\,\gamma_{c_j,b_i, e_l}=\left\{\begin{array}{cc}
 \inf \{D_{c_j}\cap D_{b_i}\cap D_{e_l}\},\ & \mbox{ if } \mu(D_{c_j}\cap D_{b_i}\cap D_{e_l})>0\\
 \beta +1,         &   \mbox{otherwise }.
\end{array}  \right.$

Similarly, for $t\in P_1$, $(Bx)(t)=b_1e_1 tI_{D_{b_1}\cap D_{e_1}}(t)-\gamma_{e_1}e_1b_1 I_{D_{b_1}}(t)$.
We now compute $(BA^n)x(t)$. By applying dominated convergence theorem, we have for $n=2$ and for
$x(\cdot)=1$  and $t\in D_{a_1} \cap D_{b_1} \cap D_{c_1}\cap D_{e_1}$,
\begin{eqnarray*}
&&  (BA^2x)(t)= b_1 I_{D_{b_1}}(t)\int\limits_{\alpha}^{t}
 e_1 I_{D_{e_1}}(s_1)ds_1\cdot
\big(   \int\limits_{\alpha}^{s_1}  a_{1}c_1
 I_{D_{a_{1}}}(s_{1})
  I_{D_{c_1}}(s_{2})ds_{2}\big)\\
 && \resizebox{0.97\hsize}{!}{$\displaystyle
 \cdot \int\limits_{\alpha}^{s_{2}} a_{1}c_1 I_{D_{a_{1}}}(s_{2})
  I_{D_{c_1}}(s_{3})ds_{3}
   =
  b_1 c^2_1a^2_1 e_1 I_{D_{b_1}}(t) \int\limits_{\alpha}^t I_{D_{e_1}}(s_1)I_{D_{a_1}}(s_1)\big(\frac{s^2_1}{2} I_{D_{a_1}\cap D_{c_1}}(s_1)
  $}
  \\
&&  - \frac{\gamma^2_{c_1,a_1}}{2}I_{D_{a_1}\cap D_{c_1}}(\gamma_{c_1,a_1})-\gamma_{c_1}s_1 I_{D_{a_1}\cap  D_{c_1}}(s_1)+\gamma_{c_1}\cdot \gamma_{c_1,a_1}I_{D_{a_1}\cap D_{c_1}}(\gamma_{c_1,a_1})\big)ds_1
\\
&&\  =
  b_1 c^2_1a^2_1 e_1 I_{D_{b_1}}(t)\big(\frac{t^3}{3!}I_{D_{a_1}\cap D_{e_1}\cap D_{c_1} }(t)
  -\frac{\gamma_{a_1,c_1,e_1}^3}{3!}I_{D_{a_1}\cap D_{e_1}\cap D_{c_1} }(\gamma_{a_1,c_1,e_1})
 \\
 && \hspace{1cm}
  -\frac{t^2}{2!}\gamma_{c_1} I_{D_{a_1}\cap D_{e_1} \cap D_{c_1}}(t)+\gamma_{c_1}\frac{\gamma^2_{a_1,c_1,e_1}}{2!}
    I_{D_{a_1}\cap D_{e_1}\cap  D_{c_1}}(\gamma_{a_1,c_1,e_1})
    \\
  &&\hspace{2cm}
  + t\gamma_{c_1}\gamma_{c_1,a_1}I_{D_{a_1}\cap D_{c_1}}(\gamma_{c_1,a_1})I_{D_{a_1}\cap D_{e_1}\cap D_{c_1}}(t)\\
  &&\hspace{3cm}
   -\gamma_{c_1} \gamma_{c_1,a_1}\gamma_{c_1,e_1}I_{D_{a_1}\cap D_{c_1}}(\gamma_{c_1,a_1})I_{D_{c_1}\cap D_{e_1}}(\gamma_{c_1,e_1}) \big)
 \end{eqnarray*}
 where $\gamma_{c_j}=\inf D_{c_j}$, $\,\gamma_{a_k,c_j}=\left\{\begin{array}{cc}
 \inf \{D_{a_k}\cap D_{c_j}\},\ & \mbox{ if } \mu(D_{a_k}\cap D_{c_j})>0\\
 \beta +1,         &   \mbox{otherwise },
\end{array}  \right.$ \\
$\,\gamma_{c_j,e_l}=\left\{\begin{array}{cc}
 \inf \{D_{c_j}\cap D_{e_l}\},\ & \mbox{ if } \mu(D_{c_j}\cap D_{e_l})>0\\
 \beta +1,         &   \mbox{otherwise },
\end{array}  \right.$\\
 $\gamma_{a_k,c_j, e_l}=\left\{\begin{array}{cc}
 \inf \{D_{a_k}\cap D_{c_j}\cap D_{e_l}\},\ & \mbox{ if } \mu(D_{a_k}\cap D_{c_j}\cap D_{e_l})>0\\
 \beta +1,         &   \mbox{otherwise }.
\end{array}  \right.$

Therefore, for $n\ge 2$, $x(\cdot)=1$  and $t\in P_1=D_{a_1} \cap D_{b_1} \cap D_{c_1}\cap D_{e_1}$ we get
\begin{multline*}
\hspace{-4mm} (BA^nx)(t) = b_1 I_{D_{b_1}}(t)\hspace{-1mm}\int\limits_{\alpha}^{t} \hspace{-2mm}
 e_1 I_{D_{e_1}}(s_1)ds_1\big( \prod\limits_{m=1}^{n-1}   \int\limits_{\alpha}^{s_m}  a_1c_1  I_{D_{a_{1}}}(s_{m})
  I_{D_{c_1}\cap D_{a_{1}}}(s_{m+1})ds_{m+1}\big) \\
 \hspace{2cm} \cdot\int\limits_{\alpha}^{s_{n}}  a_1c_1 I_{D_{a_{1}}}(s_{n}) I_{D_{c_1}}(s_{n+1})ds_{n+1}
\\
  = b_1e_1a^n_1c^n_1\big(\frac{t^{n+1}}{(n+1)!}I_{D_{b_1}\cap D_{a_1}\cap D_{c_1}\cap D_{e_1}}(t) + I_{D_{b_1}}(t)\Theta_{i,j,k,l}(t)\big),
\end{multline*}
where $\Theta_{i,j,k,l}(t)$ is a polynomial of degree $n$. Thus,
\begin{align*}
   & BF(A)x(t)=\sum_{n=0}^{\deg(F)} b_1e_1a^n_1c^n_1\big(\frac{t^{n+1}}{(n+1)!}I_{D_{b_1}\cap D_{a_1}\cap D_{c_1}\cap D_{e_1}}(t) + I_{D_{b_1}}(t)\Theta_{i,j,k,l,n}(t)\big)
   \\
   &
   =\sum_{n=0}^{\deg(F)} \big(\frac{t^{n+1}}{(n+1)!}b_1e_1a^n_1c^n_1I_{D_{b_1}\cap D_{a_1}\cap D_{c_1}\cap D_{e_1}}(t) + b_1e_1a^n_1c^n_1I_{D_{b_1}}(t)\Theta_{i,j,k,l,n}(t)\big),
\end{align*}
for almost every $t\in P_1=D_{a_1} \cap D_{b_1} \cap D_{c_1}\cap D_{e_1}$, where $\Theta_{i,j,k,l,n}(t)$ is a polynomial of degree $n$.
Thus, if $\deg(F)\ge 2$ and $(AB)x(t)=(BF(A))x(t)$ is satisfied for almost every $t\in D_{a_1} \cap D_{b_1} \cap D_{c_1}\cap D_{e_1}$, then
$\delta_{\deg(F)} b_1e_1a^n_1c^n_1I_{D_{b_1}\cap D_{a_1}\cap D_{c_1}\cap D_{e_1}}(t)=0$. Since $\delta_{\deg(F)}\not=0$, this implies
$ b_1e_1a_1c_1I_{D_{b_1}\cap D_{a_1}\cap D_{c_1}\cap D_{e_1}}(t)=0$.

Suppose that $m\ge 2$, $\sum\limits_{i,j,k,l=1}^{\infty} a_kb_ic_je_l I_{D_{a_k}\cap D_{b_i}\cap D_{c_j}\cap D_{e_l}}(t)=0$, for almost all $t\in P_r$, for all $r=1,\ldots, m$.
  By applying the dominated convergence theorem  and setting $t_m=\sup P_m$, we have for almost every $t\in P_{m+1}$,
\begin{align*}
& \hspace{0mm}(ABx)(t)=\hspace{0mm}
\sum_{k,i,j,l=1}^{\infty} \varphi(\gamma_{c_j,b_i,e_l},\gamma_{e_l})I_{\{\gamma_{c_j,b_i,e_l}\leq t_m\}}(t)I_{\{\gamma_{e_l}\leq t_m\}}(t)I_{D_{a_k}}(t)\\
&
\hspace{2cm}
+ \lim_{n_a,n_c\to\infty} \sum_{k=1}^{n_a} \sum_{j=1}^{n_c}\int\limits_{t_{m}}^{t}  a_kc_j I_{D_{a_k}}(t)
  I_{D_{c_j}}(s_1) \\
&
\hspace{3cm} \cdot\big( \lim_{n_b,n_e\to\infty}\sum_{i=1}^{n_b} \sum_{l=1}^{n_e} b_ie_l (s_1I_{D_{e_l}}(s_1)-\gamma_{e_l})
 I_{D_{b_i}}(s_1) ds_1\big)\\
 &
=\sum_{k,i,j,l=1}^{\infty} \varphi(\gamma_{c_j,b_i,e_l},\gamma_{e_l})I_{\{\gamma_{c_j,b_i,e_l}\leq t_m\}}(t)I_{\{\gamma_{e_l}\leq t_m\}}(t)I_{D_{a_k}}(t)\\
&
+\lim_{n_a,n_b,n_c, n_e\to\infty} \sum_{k=1}^{n_a} \sum_{j=1}^{n_c}\sum_{i=1}^{n_b} \sum_{l=1}^{n_e}
a_kc_jb_ie_l
\big(\frac{t^2}{2} I_{D_{b_i}\cap D_{e_l}\cap D_{a_k}\cap D_{c_j} }(t)
\\
&
-\frac{\gamma_{c_j,b_i,e_l}^2}{2}I_{D_{e_l}\cap D_{c_j}\cap D_{b_i}}(\gamma_{c_j,b_i,e_l})I_{D_{a_k}}(t)
- t\cdot \gamma_{e_l}I_{D_{b_i}\cap D_{e_l}\cap D_{a_k}\cap D_{c_j} }(t)\\
&  +\gamma_{e_l}
\gamma_{c_j,b_i,e_l}I_{D_{e_l}\cap D_{c_j}\cap D_{b_i}}(\gamma_{c_j,b_i,e_l})I_{D_{a_k}}(t)\big)
\\
&
=
\sum_{k,i,j,l=1}^{\infty} \varphi(\gamma_{c_j,b_i,e_l},\gamma_{e_l})I_{\{\gamma_{c_j,b_i,e_l}\leq t_m\}}(t)I_{\{\gamma_{e_l}\leq t_m\}}(t)I_{D_{a_k}}(t)
\\
&
+
\sum_{k,i,j,l=1}^{\infty} a_kc_jb_ie_l
\big(\frac{t^2}{2} I_{D_{b_i}\cap D_{e_l}\cap D_{a_k}\cap D_{c_j} }(t)
-\frac{\gamma_{c_j,b_i,e_l}^2}{2}I_{D_{e_l}\cap D_{c_j}\cap D_{b_i}}(\gamma_{c_j,b_i,e_l})
\\
&
\cdot I_{D_{a_k}}(t)
- t\cdot \gamma_{e_l}I_{D_{b_i}\cap D_{e_l}\cap D_{a_k}\cap D_{c_j} }(t)
  +\gamma_{e_l}
\gamma_{c_j,b_i,e_l}I_{D_{e_l}\cap D_{c_j}\cap D_{b_i}}(\gamma_{c_j,b_i,e_l})I_{D_{a_k}}(t)\big),
\end{align*}
where $ \varphi(\gamma_{c_j,b_i,e_l},\gamma_{e_l})$ is a real valued function.

Similarly for $t\in P_{m+1}$, $m\ge 2$ we have 
\begin{align*}
& \resizebox{0.98\hsize}{!}{$\displaystyle
(Bx)(t)=\sum_{i,l=1}^{\infty} 
\xi(\gamma_{e_l})I_{\{\gamma_{e_l}\leq t_m\}}(t)I_{D_{b_i}}(t)
+\lim_{n_b,n_e\to \infty}\sum_{i=1}^{n_b} \sum_{l=1}^{n_e} b_i I_{D_{b_i}}(t)\int\limits_{t_{m}}^{t}
 e_l I_{D_{e_l}}(s_1)ds_1 $}
 \\
 & \resizebox{0.98\hsize}{!}{$\displaystyle
 =\sum_{i,l=1}^{\infty} 
  \xi(\gamma_{e_l})I_{\{\gamma_{e_l}\leq t_m\}}(t)I_{D_{b_i}}(t)
 + \lim_{n_b,n_e\to \infty}\sum_{i=1}^{n_b} \sum_{l=1}^{n_e}  b_ie_l I_{D_{b_i}}(t)(t I_{D_{e_l}}(t)-\gamma_{e_l}I_{D_{e_l}}(\gamma_{e_l}))
 $}
 \\
 &
  =\sum_{i,l=1}^{\infty} 
  \xi(\gamma_{e_l})I_{\{\gamma_{e_l}\leq t_m\}}(t)I_{D_{b_i}}(t)
 + \sum_{i,l=1}^{\infty} b_ie_l I_{D_{b_i}}(t)(t I_{D_{e_l}}(t)-\gamma_{e_l}I_{D_{e_l}}(\gamma_{e_l}))
\end{align*}
where $\xi(\cdot)$ is a real valued function.

For $t\in P_{m+1}$, $m\ge 2$, we have
\begin{align*}
& \resizebox{1\hsize}{!}{$\displaystyle  (BA^2x)(t)=\sum_{k,i,j,l=1}^{\infty} \psi(\gamma_{a_k,c_j,e_l},\gamma_{a_k,c_j},\gamma_{c_j})I_{\{\gamma_{a_k,c_j,e_l}\leq t_m\}}(t)I_{\{\gamma_{a_k,c_j}\leq t_m\}}(t)I_{\{\gamma_{c_j}\leq t_m\}}(t)I_{D_{b_i}}(t)
$}
\\
&
\hspace{2cm}
+\lim_{n_b,n_e\to \infty}\sum_{i=1}^{n_b} \sum_{l=1}^{n_e} b_i I_{D_{b_i}}(t)\int\limits_{t_{m}}^{t}
 e_l I_{D_{e_l}}(s_1)ds_1\\
 &
\hspace{3cm}
 \cdot
\big(   \int\limits_{\alpha}^{s_1}\lim_{n_a\to\infty} \lim_{n_c\to\infty} \sum_{k=1}^{n_a}\sum_{j=1}^{n_c} a_{k}c_j
 I_{D_{a_{k}}}(s_{1})
  I_{D_{c_j}}(s_{2})ds_{2}\big)
  \\
  &
 \hspace{4cm}
  \cdot
  \int\limits_{\alpha}^{s_{2}} \lim_{n_a\to\infty}\lim_{n_c\to\infty} \sum_{k=1}^{n_a}\sum_{j=1}^{n_c} a_{k}c_j I_{D_{a_{k}}}(s_{2})
  I_{D_{c_j}}(s_{3})ds_{3}
  \\
& =
\sum_{k,i,j,l=1}^{\infty} \psi(\gamma_{a_k,c_j,e_l},\gamma_{a_k,c_j},\gamma_{c_j})I_{\{\gamma_{a_k,c_j,e_l}\leq t_m\}}(t)I_{\{\gamma_{a_k,c_j}\leq t_m\}}(t)I_{\{\gamma_{c_j}\leq t_m\}}(t)I_{D_{b_i}}(t)
\\
&
+\lim_{n_b,n_a,n_c,n_e\to \infty} \sum_{k=1}^{n_a}\sum_{j=1}^{n_c}\sum_{i=1}^{n_b}\sum_{l=1}^{n_e}
  b_i c^2_ja^2_k e_l I_{D_{b_i}}(t) \int\limits_{\alpha}^t I_{D_{e_l}}(s_1)I_{D_{a_k}}(s_1)\big(\frac{s^2_1}{2} I_{D_{a_k}\cap D_{c_j}}(s_1)\\
&  -\frac{\gamma^2_{c_j,a_k}}{2}I_{D_{a_k}\cap D_{c_j}}(\gamma_{c_j,a_k})-\gamma_{c_j}s_1 I_{D_{a_k}\cap  D_{c_j}}(s_1)+\gamma_{c_j}\cdot \gamma_{c_j,a_k}I_{D_{a_k}\cap D_{c_j}}(\gamma_{c_j,a_k})\big)ds_1
\\
&
=\sum_{k,i,j,l=1}^{\infty} \psi(\gamma_{a_k,c_j,e_l},\gamma_{a_k,c_j},\gamma_{c_j})I_{\{\gamma_{a_k,c_j,e_l}\leq t_m\}}(t)I_{\{\gamma_{a_k,c_j}\leq t_m\}}(t)I_{\{\gamma_{c_j}\leq t_m\}}(t)I_{D_{b_i}}(t)
\\
& \hspace{0.5cm}+\lim_{n_b,n_a,n_c,n_e\to \infty} \sum_{k=1}^{n_a}\sum_{j=1}^{n_c}\sum_{i=1}^{n_b}\sum_{l=1}^{n_e}
  b_i c^2_ja^2_k e_l I_{D_{b_i}}(t)\big(\frac{t^3}{3!}I_{D_{a_k}\cap D_{e_l}\cap D_{c_j} }(t)\\
  &
  \hspace{1.0cm}
  -\frac{\gamma_{a_k,c_j,e_l}^3}{3!}I_{D_{a_k}\cap D_{e_l}\cap D_{c_j} }(\gamma_{a_k,c_j,e_l})
  -\frac{t^2}{2!}\gamma_{c_j} I_{D_{a_k}\cap D_{e_l} \cap D_{c_j}}(t)+\gamma_{c_j}\frac{\gamma^2_{a_k,c_j,e_l}}{2!}
  \\
   &
   \hspace{1.5cm}
   \cdot I_{D_{a_k}\cap D_{e_l}\cap  D_{c_j}}(\gamma_{a_k,c_j,e_l})+t\gamma_{c_j}\gamma_{c_j,a_k}I_{D_{a_k}\cap D_{c_j}}(\gamma_{c_j,a_k})I_{D_{a_k}\cap D_{e_l}\cap D_{c_j}}(t)\\
  &
 \hspace{2.5cm} -\gamma_{c_j} \gamma_{c_j,a_k}\gamma_{c_j,e_l}I_{D_{a_k}\cap D_{c_j}}(\gamma_{c_j,a_k}I_{D_{c_j}\cap D_{e_l}}(\gamma_{c_j,e_l}) \big),
 \end{align*}
 where $ \psi(\gamma_{a_k,c_j,e_l},\gamma_{a_k,c_j} \gamma_{c_j})$ is a real valued function.
For $n\ge 2$, $t\in P_{m+1}$,
\begin{align*}
&
(BA^nx)(t)
\\
& =
\sum_{k,i,j,l=1}^{\infty} \psi(\gamma_{a_k,c_j,e_l},\gamma_{a_k,c_j},\gamma_{c_j})I_{\{\gamma_{a_k,c_j,e_l}\leq t_m\}}(t)I_{\{\gamma_{a_k,c_j}\leq t_m\}}(t)I_{\{\gamma_{c_j}\leq t_m\}}(t)
 I_{D_{b_i}}(t)\\
 &
\hspace{1cm} +\lim_{n_b,n_e\to \infty}\sum_{i=1}^{n_b} \sum_{l=1}^{n_e} b_i I_{D_{b_i}}(t)\hspace{0mm}\int\limits_{t_{m}}^{t}
 e_l I_{D_{e_l}}(s_1)ds_1\\
 &
\hspace{2cm} \cdot \big( \prod\limits_{m=1}^{n-1} \lim_{n_a,n_c\to\infty}\sum_{k=1}^{n_a} \sum_{j=1}^{n_c}  \int\limits_{\alpha}^{s_m}  a_kc_j I_{D_{a_{k}}}(s_{m}) I_{D_{c_j}\cap D_{a_{k}}}(s_{m+1})ds_{m+1}\big)
 \\
 &
\hspace{3cm} \cdot \lim_{n_a, n_c\to\infty}
  \int\limits_{\alpha}^{s_{n}} \sum_{k=1}^{n_a} \sum_{j=1}^{n_c} a_kc_j I_{D_{a_{k}}}(s_{n}) I_{D_{c_j}}(s_{n+1})ds_{n+1}
  \\
  & =
  \sum_{k,i,j,l=1}^{\infty} \psi(\gamma_{a_k,c_j,e_l},\gamma_{a_k,c_j},\gamma_{c_j})I_{\{\gamma_{a_k,c_j,e_l}\leq t_m\}}(t)I_{\{\gamma_{a_k,c_j}\leq t_m\}}(t)I_{\{\gamma_{c_j}\leq t_m\}}(t)I_{D_{b_i}}(t)
  \\
  &
  \hspace{1cm} +
  \sum_{i,j,k,l=1}^{\infty} b_ie_la^n_kc^n_j\big(\frac{t^{n+1}}{(n+1)!}I_{D_{b_i}\cap D_{a_k}\cap D_{c_j}\cap D_{e_l}}(t) + I_{D_{b_i}}(t)\Theta_{i,j,k,l}(t)\big),
\end{align*}
where $\Theta_{i,j,k,l}(t)$ is a polynomial of degree $n$. Therefore,
\begin{align*}
&  \resizebox{0.89\hsize}{!}{$\displaystyle (BF(A)x)(t) =\hspace{0mm}
  \sum_{n=0}^{\deg(F)}\hspace{0mm}\sum_{k,i,j,l=1}^{\infty} \psi(\gamma_{a_k,c_j,e_l},\gamma_{a_k,c_j},\gamma_{c_j},\gamma_{e_l},n)I_{\{\gamma_{a_k,c_j,e_l}\leq t_m\}}(t)
  $}
  \\
  & 
 \hspace{4cm} \cdot I_{\{\gamma_{a_k,c_j}\leq t_m\}}(t)
   I_{\{\gamma_{c_j}\leq t_m\}}(t)I_{\{\gamma_{e_l}\leq t_m\}}(t)I_{D_{b_i}}(t)
  \\
  &
  +
  \sum_{n=0}^{\deg(F)}\sum_{i,j,k,l=1}^{\infty}\delta_n b_ie_la^n_kc^n_j\big(\frac{t^{n+1}}{(n+1)!}
  I_{D_{b_i}\cap D_{a_k}\cap D_{c_j}\cap D_{e_l}}(t)
   + I_{D_{b_i}}(t)\Theta_{i,j,k,l,n}(t)\big)
   \\
   &=
  \sum_{n=0}^{\deg(F)}\hspace{-1mm}\sum_{k,i,j,l=1}^{\infty} \psi(\gamma_{a_k,c_j,e_l},\gamma_{a_k,c_j},\gamma_{c_j},\gamma_{e_l},n)
  \\
  &
\hspace{2cm} \cdot I_{\{\gamma_{a_k,c_j,e_l}\leq t_m\}}(t)I_{\{\gamma_{a_k,c_j}\leq t_m\}}(t)
  I_{\{\gamma_{c_j}\leq t_m\}}(t)I_{\{\gamma_{e_l}\leq t_m\}}(t)I_{D_{b_i}}(t)
  \\
  & \resizebox{0.97\hsize}{!}{$\displaystyle
 + \sum_{n=0}^{\deg(F)}\big(\frac{t^{n+1}}{(n+1)!}\sum_{i,j,k,l=1}^{\infty}\delta_n b_ie_la^n_kc^n_jI_{D_{b_i}\cap D_{a_k}\cap D_{c_j}\cap D_{e_l}}(t)
   + \sum_{i,j,k,l=1}^{\infty}\delta_n b_ie_la^n_kc^n_jI_{D_{b_i}}(t)\Theta_{i,j,k,l,n}(t)\big),
  $}
\end{align*}
where $\Theta_{i,j,k,l,n}(t)$ is a polynomial of degree $n$, $\psi(\gamma_{a_k,c_j,e_l},\gamma_{a_k,c_j},\gamma_{c_j}, \gamma_{e_l},n)$  is a real valued sequence of functions.

Therefore, if $\deg(F)\ge 2$ and $(AB)x(t)=(BF(A))x(t)$ for almost every $t\in P_{m+1}$, then
$\delta_{\nu}\cdot\sum\limits_{i,j,k,l=1}^{\infty} b_ie_la^\nu_kc^\nu_jI_{D_{b_i}\cap D_{a_k}\cap D_{c_j}\cap D_{e_l}}(t)=0$, $t\in P_{m+1}$, where $\nu=\deg(F)$. Since $\delta_{\nu}\not=0$ this implies
$
\sum\limits_{i,j,k,l=1}^{\infty} b_ie_la^\nu_kc^\nu_jI_{D_{b_i}\cap D_{a_k}\cap D_{c_j}\cap D_{e_l}}(t)=0
$, $t\in P_{m+1}$, which is equivalent to
\begin{align*}
&\sum\limits_{i,j,k,l=1}^{\infty} b_ie_la_kc_jI_{D_{b_i}\cap D_{a_k}\cap D_{c_j}\cap D_{e_l}}(t)=
\sum\limits_{k=1}^{\infty}a_k I_{D_{a_k}}(t) \sum\limits_{i=1}^{\infty} b_i I_{D_{b_i}}(t) \sum\limits_{j=1}^{\infty} c_j I_{D_{c_j}}(t)
\\
&
\hspace{2cm}
\cdot \sum\limits_{l=1}^{\infty} e_l I_{D_{e_l}}(t)=a(t)b(t)c(t)e(t)=0,
\end{align*}
 for almost every $t\in P_{m+1}$. Since $m$ is any positive integer, we conclude that  $(AB)x(t)=(BF(A))x(t)$, $t\in[\alpha,\beta]$ implies that the set supp$\{abce\}$ has measure zero.
\qed
\end{proof}

\begin{corollary}
Let $1\leq p\leq\infty$, and the linear operators $A:\,L_p(\mathbb{[\alpha,\beta]})\to L_p(\mathbb{[\alpha,\beta]})$ and  $B:\,L_p(\mathbb{[\alpha,\beta]})\to L_p(\mathbb{[\alpha,\beta]})$ be defined, for almost every $t$, by
\begin{gather*}
  (Ax)(t)= \int\limits_{\alpha}^{t} \gamma_A x(s)ds,\quad (Bx)(t)= \int\limits_{\alpha}^{t} \gamma_B x(s)ds,
\end{gather*}
where $\alpha,\gamma_A,\gamma_B$ are  real numbers. Let $F(z)=\sum\limits_{i=0}^{n}\delta_i z^i$, $\delta_i\in\mathbb{R}$, $i=0,\ldots, n$,  $n\ge 2$.

If $AB=BF(A)$, then either $A=0$ or $B=0$. In particular, if $A=0$ and $F(0)\not=0$, then $B=0$. Moreover, if $F(0)=0$, that is, $F(z)=\sum\limits_{i=1}^{n}\delta_i z^i$, then $AB=BF(A)$ if and only if either $A=0$ or $B=0$.
\end{corollary}

\begin{proof}
 If $AB= BF(A)$, $n\ge 2$ then by applying Proposition \ref{propOpIntVoltTypeSimple} we get $\gamma_A\gamma_B=0$. This implies that either $\gamma_A=0$ or $\gamma_B=0$. From which implies $A=0$ or $B=0$. If $A=0$ then
 commutation relation $AB=BF(A)$ implies $BF(0)=0$. This implies that $B=0$ if $F(0)\not=0$.

\noindent In case $F(0)=0$, that is, $F(t)=\sum\limits_{i=1}^{n}\delta_i t^i$ then, if  $AB=BF(A)$ then by applying Proposition \ref{propOpIntVoltTypeSimple} we get $\gamma_A\gamma_B=0$. Which implies that either $A=0$ or $B=0$.
Conversely if either $A=0$ or $B=0$ the commutation relation is satisfied.
  \qed
\end{proof}

\begin{example}
Let $1<p<\infty$ and the linear operators
$A:L_p[0,1]\to L_p[0,1]$ and $B:L_p[0,1]\to L_p[0,1]$,  be defined, for almost every $t$, by
\begin{equation*}
   (Ax)(t)=\int\limits_{0}^{t} a(t)x(s)ds, \quad  (Bx)(t)=\int\limits_{0}^{t} b(t)x(s)ds,
\end{equation*}
where $a(t)=I_{[0,1/2]}(t)$, $b(t)=I_{[1/2,1]}(t)$, $I_\Lambda$ is the indicator function of the set $\Lambda$, $c(\cdot)=e(\cdot)=1$ in the notation of Proposition \ref{propOpIntVoltTypeSimple}.
By computing $AB$ we get
\begin{equation*}
  (AB)x(t)=\int\limits_0^t a(t)\big(\int\limits_{0}^{s}b(s)x(\tau)d\tau\big) ds=a(t)\int\limits_0^t b(s)\big(\int\limits_{0}^{s}x(\tau)d\tau\big) ds=0,
\end{equation*}
for almost every $t$, because if $0\le t<1/2$ then $b(s)=0$ for almost every $s\in [0,1/2]$ and if
$1/2\le t<1$ then $a(t)=0$ for almost every $t\in [1/2,1]$. We compute
$A^2$ and $BA^2$. For almost every $0<t<1/2$,
\begin{equation*}
   (A^2x)(t)  =\int\limits_{0}^{t} a(t)\big(\int\limits_{0}^{s} a(s)x(\tau)d\tau\big) ds=\int\limits_{0}^{t} \big(\int\limits_{0}^{s}x(\tau)d\tau\big) ds
\end{equation*}
and  for almost every $1/2\le t\le 1$,
\begin{align*}
  & (A^2x)(t)  =\int\limits_{0}^{1/2} a(t)\big(\int\limits_{0}^{s} a(s)x(\tau)d\tau\big) ds+\int\limits_{1/2}^{t} a(t)\big(\int\limits_{0}^{s}a(s)x(\tau)d\tau\big) ds
 \\
 &
 = \int\limits_{0}^{1/2} a(t)\big(\int\limits_{0}^{s} a(s)x(\tau)d\tau\big) ds=0, 
\end{align*}
because $a(t)=0$ for $1/2 \leq t \leq 1$. Therefore,  $(BA^2x)(t)=0$.
So, we have  $AB=\delta BA^2=0$, and $a(t)b(t)c(t)e(t)=0$ for almost every $t$.
\end{example}

\begin{remark}
Note that the necessary condition for commutation relation $AB=\delta BA^n$, $n\ge 2$ in Proposition \ref{propOpIntVoltTypeSimple} is not sufficient. In fact, let $1<p<\infty$, and let the linear operators
$A:L_p[0,1]\to L_p[0,1]$ and $B:L_p[0,1]\to L_p[0,1]$ be defined, for almost every $t$, by
$
   (Ax)(t)=\int\limits_{0}^{t} a(t)x(s)ds,\  (Bx)(t)=\int\limits_{0}^{t} b(t)x(s)ds,
$
and where $a(t)=I_{[0,1/4]}(t)-I_{[3/4,1]}(t)$, $b(t)=I_{[1/4,3/4]}(t)$, $I_\Lambda$ is the indicator function of the set $\Lambda$, and $c(\cdot)=e(\cdot)=1$ in the notation of Proposition \ref{propOpIntVoltTypeSimple}.
Let us compute $AB$. If $0\leq t \leq 1/4$, then for almost every $t$,
\begin{equation*}
\resizebox{0.98\hsize}{!}{$\displaystyle  (AB)x(t)=\int\limits_0^t a(t)\big(\int\limits_{0}^{s}b(s)x(\tau)d\tau\big) ds=a(t)\int\limits_0^t b(s)\big(\int\limits_{0}^{s}x(\tau)d\tau\big) ds=0, $}
\end{equation*}
because if $0 \leq t <1/4$, then $b(s)=0$ for almost every $s\in [0,1/4]$. If  $1/4\leq t \leq 3/4$, then
\begin{eqnarray*}
 (AB)x(t)&=&\int\limits_0^{1/4} a(t)\big(\int\limits_{0}^{s}b(s)x(\tau)d\tau\big) ds+
  \int\limits_{1/4}^{t} a(t)\big(\int\limits_{0}^{s}b(s)x(\tau)d\tau\big) ds\\
 &=&  a(t)\int\limits_{1/4}^t b(s)\big(\int\limits_{0}^{s}x(\tau)d\tau\big) ds=0,
\end{eqnarray*}
for almost every $t$. If  $3/4\leq t \leq 1$, then
\begin{eqnarray*}
 && (AB)x(t)=\int\limits_0^{1/4} a(t)\big(\int\limits_{0}^{s}b(s)x(\tau)d\tau\big) ds
 + \int\limits_{1/4}^{3/4} a(t)\big(\int\limits_{0}^{s}b(s)x(\tau)d\tau\big) ds\\
 && + \int\limits_{3/4}^{t} a(t)\big(\int\limits_{0}^{s}b(s)x(\tau)d\tau\big) ds=
  -I_{[3/4,1]}(t)\int\limits_{1/4}^{3/4} \big(\int\limits_{0}^{s}x(\tau)d\tau\big) ds
\end{eqnarray*}
for almost every $t$. We compute
$A^2$ and $BA^2$. For almost every $0<t<1/4$,
\begin{equation*}
   (A^2x)(t)  =\int\limits_{0}^{t} a(t)\big(\int\limits_{0}^{s} a(s)x(\tau)d\tau\big) ds=\int\limits_{0}^{t} \big(\int\limits_{0}^{s}x(\tau)d\tau\big) ds.
\end{equation*}
For almost every $1/4\le t\le 3/4$,
\begin{eqnarray*}
   (A^2x)(t) & =&\int\limits_{0}^{1/4} a(t)\big(\int\limits_{0}^{s} a(s)x(\tau)d\tau\big) ds+\int\limits_{1/4}^{t} a(t)\big(\int\limits_{0}^{s}a(s)x(\tau)d\tau\big) ds\\
 &=&I_{[0,1/4]}(t)\int\limits_{0}^{1/4} \big(\int\limits_{0}^{s} x(\tau)d\tau\big) ds=0.
\end{eqnarray*}
For almost every $3/4\le t\le 1$,
\begin{eqnarray*}
   (A^2x)(t) & =&\int\limits_{0}^{1/4} a(t)\big(\int\limits_{0}^{s} a(s)x(\tau)d\tau\big) ds+\int\limits_{3/4}^{t} a(t)\big(\int\limits_{0}^{s}a(s)x(\tau)d\tau\big) ds\\
  &=&-I_{[3/4;1]}(t)\big(\int\limits_{0}^{1/4} \big(\int\limits_{0}^{s} x(\tau)d\tau\big) ds-\int\limits_{3/4}^{t} \big(\int\limits_{0}^{s} x(\tau)d\tau\big) ds\big).
\end{eqnarray*}
Therefore, for almost every $t\in [0,1/4]$, we get $(BA^2)x(t)=0$, since $b(t)=0$ for almost every $t$ in this interval. For almost every $t\in [1/4;3/4]$,
\begin{align*}
 & \hspace{0mm} (BA^2x)(t)=\int\limits_0^{1/4} b(t)\big(\int\limits_{0}^{\tau} \big(\int\limits_{0}^{s_1} a(s_1)x(s_2)ds_2\big) ds_1\big) d\tau \\
 & 
\hspace{4cm}  + \int\limits_{1/4}^{t} b(t)\hspace{0mm}\big(\hspace{0mm}\int\limits_{0}^{\tau} \big(\int\limits_{0}^{s_1} a(s_1)x(s_2)ds_2 \big) ds_1 \hspace{0mm}\big) d\tau\\
 &\resizebox{1\hsize}{!}{$\displaystyle
  =I_{[1/4;3/4]}(t)\int\limits_0^{1/4}\hspace{0mm} \big(\hspace{0mm}\int\limits_{0}^{\tau} \big( \int\limits_{0}^{s_1} a(s_1)x(s_2)ds_2\big) ds_1 \hspace{-1mm} +\int\limits_{1/4}^{t} b(t)\hspace{0mm}\big(\hspace{0mm}\int\limits_{0}^{\tau} \big(\int\limits_{0}^{s_1} a(s_1)x(s_2)ds_2 \big) ds_1 \hspace{0mm}\big) d\tau\big) d\tau. $}
\end{align*}
For almost every $t\in [3/4,1]$, we get $(BA^2)x(t)=0$, since $b(t)=0$ for almost every $t$ in this interval.
So, we have $a(t)b(t)c(t)e(t)=0$ for almost every $t$, but $AB\not=BA^2$.
\end{remark}

\begin{proposition}\label{PropSufficientCondVolteraOpMonomialComutRelation}
Let  $1\leq p\leq\infty$, $1\leq q\leq\infty$ be such that $\frac{1}{p}+\frac{1}{q}=1$, and let the linear operators $A:\,L_p(\mathbb{[\alpha,\beta]})\to L_p(\mathbb{[\alpha,\beta]})$ and  $B:\,L_p(\mathbb{[\alpha,\beta]})\to L_p(\mathbb{[\alpha,\beta]})$ be defined, for almost every $t$, by
\begin{gather*} 
  (Ax)(t)= \int\limits_{\alpha}^{t} a(t)c(s) x(s)ds,\quad (Bx)(t)= \int\limits_{\alpha}^{t} b(t)e(s) x(s)ds,
\end{gather*}
where $\alpha,\beta $ are real numbers, $\alpha<\beta$, $a,b\in L_p([\alpha,\beta])$ and $c,e\in L_q([\alpha,\beta])$. Let $F$ be a polynomial
$F(z)=\sum\limits_{n=1}^{\deg(F)} \delta_n z^n$, with $\delta_n\in \mathbb{R}$ for $n=1,\ldots,\deg(F)$.
If $ae=0$ almost everywhere and $bc=0$ almost everywhere, then $AB=BF(A)=0$.
\end{proposition}

\begin{proof}
By applying H\"older inequality we can conclude that operators are well defined and bounded.
Computing $AB$, we get for all $x\in L_p[\alpha,\beta]$, and almost every $t$,
\begin{align*}
(ABx)(t) & =\int\limits_{\alpha}^{\beta} a(t)c(s)\big(\int\limits_\alpha^\beta b(s)e(\tau)x(\tau) d\tau\big) ds\\
& =\int\limits_{\alpha}^{\beta} a(t)c(s)b(s)\big(\int\limits_\alpha^\beta e(\tau)x(\tau) d\tau\big) ds=0,
\end{align*}
 since $b(t)c(t)=0$ for almost every $t$.
Computing $A^n$, we get for all $x\in L_p[\alpha,\beta]$,
\begin{align*}
  (A^2x)(t) & =\int\limits_{\alpha}^{\beta} a(t)c(s)\big(\int\limits_\alpha^\beta a(s)c(\tau)x(\tau) d\tau\big) ds\\
  & =\int\limits_{\alpha}^{\beta} a(t)c(s)a(s)\big(\int\limits_\alpha^\beta c(\tau)x(\tau) d\tau\big) ds
\end{align*}
Therefore, for $n\ge 2$ we have for all $x\in L_p[\alpha,\beta]$,  and almost every $t$,
\begin{equation*}
   (A^n x)(t) = a(t)\int\limits_{\alpha}^{t} c(s_1)\big(\prod_{i=1}^{n-1}\int\limits_{\alpha}^{s_i} a(s_{i})c(s_{i+1})x(s_{i+1})ds_{i+1}\big)ds_1.
\end{equation*}
Hence, for $n\ge 2$ we have for all $x\in L_p[\alpha,\beta]$ and almost every $t$
\begin{equation*}
  (BA^n)x(t)=b(t) \int\limits_{\alpha}^{t} e(\tau)a(\tau) \big(\prod_{i=1}^{n}\int\limits_{\alpha}^{s_i} x(s_{i+1})ds_{i+1}ds_1\big) d\tau=0,
\end{equation*}
since $a(t)e(t)=0$ for almost every $t$. Therefore,
\begin{equation*}
  (BF(A))x(t)=\sum_{n=1}^{\deg(F)} \delta_n b(t) \int\limits_{\alpha}^{t} e(\tau)a(\tau) \big(\prod_{i=1}^{n}\int\limits_{\alpha}^{s_i} x(s_{i+1})ds_{i+1}ds_1\big) d\tau=0,
\end{equation*}
for almost every $t$.
  \qed
\end{proof}

\begin{example}
Let  $1<p<\infty$, and the linear operators $A:L_p[0,1]\to L_p[0,1]$ and $B:L_p[0,1]\to L_p[0,1]$ be defined, for almost every $t$, by
\begin{align*}
   & (Ax)(t)=\int\limits_{0}^{t} a(t)c(s)x(s)ds, \quad  (Bx)(t)=\int\limits_{0}^{t} b(t)e(s)x(s)ds,\\
   & a(t)=I_{[0,1/4]}(t)-I_{[1/2,3/4]}(t),\ b(t)=I_{[1/2,1]}(t),\\
   & c(t)=I_{[0,1/2]}(t),\ e(t)=I_{[1/4,1/2]}(t)+I_{[3/4,1]}(t)
\end{align*}
where $I_\Lambda$ is the indicator function of the set $\Lambda$. We have
$a(\cdot)e(\cdot)=0$ almost everywhere and $b(t)c(t)=0$ for almost every $t$. Therefore, by applying Proposition \ref{PropSufficientCondVolteraOpMonomialComutRelation} we conclude that $AB=BA^n=0$ for all integers $n\geq 2$.
\end{example}

\begin{remark}
In Proposition \ref{propOpIntVoltTypeSimple} if $n=0$, that is, if $AB=\delta_0 B$, $\delta_0\not=0$ then
 $abce=0$ almost everywhere.

Note that $ \bigcup\limits_{i,i,k,l=1}^{\infty}\{ D_{a_k} \cap D_{b_i} \cap D_{c_j}\cap D_{e_l} \}=\bigcup\limits_{m=1}^\infty P_m $ is a partition of $[\alpha,\beta]$. Therefore, for $x(\cdot)=1$  and $t\in P_1=D_{a_1} \cap D_{b_1} \cap D_{c_1}\cap D_{e_1}$, by applying the dominated convergence theorem we have
\begin{align*}
& \hspace{0mm}(ABx)(t)=\hspace{0mm}\int\limits_{\alpha}^{t}  a_1c_1 I_{D_{a_1}}(t)
  I_{D_{c_1}}(s_1) 
\hspace{0mm}  b_1e_1 (s_1I_{D_{e_1}}(s_1)-\gamma_{e_1})I_{D_{b_1}}(s_1) ds_1\\
&=
a_1c_1b_1e_1
\big(\frac{t^2}{2} I_{D_{b_1}\cap D_{e_1}\cap D_{a_1}\cap D_{c_1} }(t)
-\frac{\gamma_{c_1,b_1,e_1}^2}{2}I_{D_{e_1}\cap D_{c_1}\cap D_{b_1}}(\gamma_{c_1,b_1,e_1})I_{D_{a_1}}(t)
\\
&
- t\cdot \gamma_{e_1}I_{D_{b_1}\cap D_{e_1}\cap D_{a_1}\cap D_{c_1} }(t)
 +\gamma_{e_l}
\gamma_{c_1,b_1,e_1}I_{D_{e_1}\cap D_{c_1}\cap D_{b_1}}(\gamma_{c_1,b_1,e_1})I_{D_{a_1}}(t)\big)
\\
&
= a_1c_1b_1e_1
\big(\frac{t^2}{2} I_{D_{b_1}\cap D_{e_1}\cap D_{a_1}\cap D_{c_1} }(t)
-\frac{\gamma_{c_1,b_1,e_1}^2}{2}I_{D_{e_1}\cap D_{c_1}\cap D_{b_1}}(\gamma_{c_1,b_1,e_1})I_{D_{a_1}}(t)
\\
&
- t\cdot \gamma_{e_1}I_{D_{b_1}\cap D_{e_1}\cap D_{a_1}\cap D_{c_1} }(t)
  +\gamma_{e_1}
\gamma_{c_1,b_1,e_1}I_{D_{e_1}\cap D_{c_1}\cap D_{b_1}}(\gamma_{c_1,b_1,e_1})I_{D_{a_1}}(t)\big),
\end{align*}
where $\gamma_{e_l}=\inf D_{e_l}$, $\,\gamma_{c_j,b_i, e_l}=\left\{\begin{array}{cc}
 \inf \{D_{c_j}\cap D_{b_i}\cap D_{e_l}\},\ & \mbox{ if } \mu(D_{c_j}\cap D_{b_i}\cap D_{e_l})>0\\
 \beta +1,         &   \mbox{otherwise }.
\end{array}  \right.$
For $t\in P_1$, $(Bx)(t)=b_1e_1 tI_{D_{b_1}\cap D_{e_1}}(t)-\gamma_{e_1}e_1b_1 I_{D_{b_1}}(t) $.
Therefore, if $(ABx)(t)=\delta_0 (Bx)(t)$, $\delta_0\not=0$, then  $a_1c_1b_1e_1
 I_{D_{b_1}\cap D_{e_1}\cap D_{a_1}\cap D_{c_1} }(t)=0$ for almost every $t$.
By applying the same procedure as in the proof of Proposition \ref{propOpIntVoltTypeSimple} we conclude that
the set supp$\{abce\}$ has measure zero.
 \end{remark}

\begin{theorem}\label{ThmVolterraOpCommutativityCond}
Let $1\leq p\leq \infty$ and the linear operators $A:\,L_p(\mathbb{[\alpha,\beta]})\to L_p (\mathbb{[\alpha,\beta]})$ and  $B:\, L_p(\mathbb{[\alpha,\beta]})\to L_p(\mathbb{[\alpha,\beta]})$ be defined, for almost every $t\in[\alpha,\beta]$, by
\begin{gather*}
  (Ax)(t)= \int\limits_{\gamma_a}^{t} k_A(t,s)x(s)ds,\quad (Bx)(t)= \int\limits_{\gamma_b}^{t} {k}_B(t,s)x(s)ds,
\end{gather*}
where $\alpha,\gamma\in \mathbb{R}$, $\alpha \leq \gamma_a\leq \gamma_b\le \beta$, and $k_A:[\alpha,\beta]^2\to\mathbb{R},\ {k}_B:[\alpha,\beta]^2\to\mathbb{R}$ are measurable functions. 
We set
\begin{eqnarray*}
  \Gamma &=&\{(t,s,\tau)\in [\gamma_b,\beta]^3:\ \gamma_b\leq t\leq \beta,\ \gamma_b\leq s\leq t,\ \gamma_b\leq \tau \leq t\},\\ 
  \Delta &=& [\gamma_b,\beta]\times[\gamma_a,\gamma_b].  
\end{eqnarray*}
Then $AB= \delta BA$ for given $\delta\in \mathbb{R}\setminus \{0\}$
if and only if
\begin{enumerate}[label=\textup{\arabic*.}, ref=\arabic*]
  \item\label{volterraQuantumPlaneRelationCondition1} for almost every  $(t,s,\tau)\in \Gamma$, we have $ k_A(t,s){k}_B(s,\tau)=\delta {k}_B(t,s)k_A(s,\tau)$;
  \item\label{volterraQuantumPlaneRelationCondition2} for almost every $(t,\tau)\in\Delta$ we have
   \begin{equation*}
     \int\limits_{\gamma_b}^{t} {k}_B(t,s)k_A(s,\tau)ds=0.
   \end{equation*}
\end{enumerate}
\end{theorem}

\begin{proof}
We compute $AB$ and $BA$ by applying Fubini Theorem (see \cite{AdamsG}) and considering that $\gamma_a\le \gamma_b$, we have
  \begin{align*}
    &(ABx)(t)=\int\limits_{\gamma_a}^{t} k_A(t,s)\big(\int\limits_{\gamma_b}^{s}{k}_B(s,\tau)x(\tau)d\tau\big)ds
    \\
    &=\int\limits_{\gamma_a}^{\gamma_b} x(\tau) \big(\int\limits_{\gamma_a}^{\tau}k_A(t,s)I_{[\gamma_b,\beta]}(\tau){k}_B(s,\tau)ds\big)d\tau+
    \int\limits_{\gamma_b}^{t} x(\tau) \big(\int\limits_{\tau}^{t}k_A(t,s){k}_B(s,\tau)ds\big)d\tau\\
    &=\int\limits_{\gamma_a}^{\gamma_b} k_{1,AB}(t,\tau) x(\tau)d\tau+\int\limits_{\gamma_b}^{t} k_{2,AB}(t,\tau) x(\tau)d\tau,
  \end{align*}
  where
  \begin{eqnarray}\label{IntervalTauKAB1volterraQuantumPlaneRelation}
    k_{1,AB}(t,\tau)&=&\int\limits_{\gamma_a}^{\tau}k_A(t,s)I_{[\gamma_b,\beta]}(\tau){k}_B(s,\tau)ds=0,\quad \alpha\leq \tau  \leq \gamma_b \\
    \label{IntervalTauKAB2volterraQuantumPlaneRelation}
    k_{2,AB}(t,\tau)&=&\int\limits_{\tau}^{t}k_A(t,s){k}_B(s,\tau)ds,\quad \gamma_b \leq \tau \leq t .
  \end{eqnarray}
   We use the same procedure to compute $BA$. We have
   \begin{eqnarray*}
    (BAx)(t)&=&\int\limits_{\gamma_b}^{t} {k}_B(t,s)\big(\int\limits_{\gamma_a}^{s}{k}_A(s,\tau)x(\tau)d\tau\big)ds
    \\
    &=&\int\limits_{\gamma_a}^{\gamma_b} x(\tau) \big(\int\limits_{\gamma_b}^{t}{k}_B(t,s){k}_A(s,\tau)ds\big)d\tau+
    \int\limits_{\gamma_b}^{t} x(\tau) \big(\int\limits_{\tau}^{t}{k}_B(t,s){k}_A(s,\tau)ds\big)d\tau\\
    &=&\int\limits_{\gamma_a}^{\gamma_b} k_{1,BA}(t,\tau) x(\tau)d\tau+\int\limits_{\gamma_b}^{t} k_{2,BA}(t,\tau) x(\tau)d\tau,
  \end{eqnarray*}
  where
   \begin{eqnarray}\label{IntervalTauK1BAvolterraQuantumPlaneRelation}
    k_{1,BA}(t,\tau)&=&\int\limits_{\gamma_b}^{t} {k}_B(t,s){k}_A(s,\tau)ds, \ \gamma_a \leq \tau \leq \gamma_b
    \\ \label{IntervalTauK2BAvolterraQuantumPlaneRelation}
    k_{2,BA}(t,\tau)&=&\int\limits_{\tau}^{t} {k}_B(t,s){k}_A(s,\tau)ds, \ \gamma_b \leq \tau \leq t.
  \end{eqnarray}
   Thus, for all $x\in L_p[\alpha,\beta]$ and given $\delta\in\mathbb{R}\setminus\{0\}$ we have $ABx=\delta BAx$  if and only if
   \begin{align}\nonumber
    & \resizebox{1.0\hsize}{!}{$\displaystyle \int\limits_{\gamma_a}^{\gamma_b} k_{1,AB}(t,\tau) x(\tau)d\tau  + \int\limits_{\gamma_b}^{t} k_{2,AB}(t,\tau) x(\tau)d\tau=
    \delta \int\limits_{\gamma_a}^{\gamma_b} k_{1,BA}(t,\tau) x(\tau)d\tau 
      +  \delta\int\limits_{\gamma_b}^{t} k_{2,BA}(t,\tau) x(\tau)d\tau $}\\
     &  \Leftrightarrow  \label{VolterraCommutationCond}
      \resizebox{0.82\hsize}{!}{$\displaystyle\int\limits_{\gamma_a}^{\gamma_b} 
      \delta k_{1,BA}(t,\tau) 
      x(\tau)d\tau =\int\limits_{\gamma_b}^{t} [-\delta k_{2,BA}(t,\tau)+k_{2,AB}(t,\tau)] x(\tau)d\tau. $}
   \end{align}
   We consider two cases.

\noindent {\it Case 1}: If $\gamma_b \leq t \leq \beta$ and $\gamma_a\le \gamma_b$, the set $[\gamma,t]\cap [\gamma_a,\gamma_b]=\{\gamma_b\}$ has measure zero. By applying Lemma \ref{lemEqIntForAllLpFunctFinitMeasureA}, we conclude that \eqref{VolterraCommutationCond} is equivalent to the following: for almost every $t>\gamma_b$
   \begin{eqnarray} \label{EqualityK1ABK1BAvolterraQuantumPlaneRelation}
    \delta k_{1,BA}(t,\tau)&=& 0,\ \mbox{ for almost every }\ \tau\in [\gamma_a,\gamma_b],  \\
    \label{EqualityK2ABK2BAvolterraQuantumPlaneRelation}
    \delta k_{2,BA}(t,\tau)&=&k_{2,AB}(t,\tau), \ \mbox{ for almost every }\ \tau\in [\gamma_b,t].
   \end{eqnarray}
   Using \eqref{IntervalTauKAB1volterraQuantumPlaneRelation}, \eqref{IntervalTauKAB2volterraQuantumPlaneRelation}, \eqref{IntervalTauK1BAvolterraQuantumPlaneRelation},   \eqref{IntervalTauK2BAvolterraQuantumPlaneRelation}, yields that
\eqref{EqualityK1ABK1BAvolterraQuantumPlaneRelation} and \eqref{EqualityK2ABK2BAvolterraQuantumPlaneRelation} are equivalent to
   \begin{gather*}
   \delta \int\limits_{\gamma_b}^{t} {k}_B(t,s)k_A(s,\tau)ds=0, 
   \ \mbox{for almost every }\ (t,\tau)\in\Delta,   \\ \vspace{0.2cm}
   \begin{array}{l} \displaystyle{\int\limits_{\tau}^{t}} k_A(t,s){k}_B(s,\tau)ds=\delta \displaystyle{\int\limits_{\tau}^{t}} {k}_B(t,s)k_A(s,\tau)ds\\
   \mbox{for almost every}\ (t,\tau)\in[\gamma_b,\beta]\times[\gamma_b,t]
   \end{array} \Leftrightarrow
   \begin{array}{l}
   k_A(t,s){k}_B(s,\tau)\\
   \hspace{1cm}=\delta {k}_B(t,s)k_A(s,\tau) \\
   \mbox{ for almost every } (t,s,\tau)\in \Gamma.
   \end{array}
   \end{gather*}

\noindent {\it Case 2}: If $\gamma_a \leq t \leq \gamma_b$, then $(Bx)(t)=0$ for all $x\in L_p[\alpha,\beta]$. Therefore, the commutation relation is satisfied.
   \qed
\end{proof}

\begin{remark}
One can give a different proof for Theorem \ref{ThmVolterraOpCommutativityCond} by noticing that $\alpha\leq\gamma_a\leq \gamma_b\leq \beta$, and that for almost every $t$,
\begin{align*}
 & (Ax)(t)=\int\limits_{\gamma_a}^{t} k_A(t,s)x(s)ds=\int\limits_{\gamma_a}^{\beta} I_{[\gamma_a,t]}(s)k_A(t,s)x(s)ds,\\
 & (Bx)(t)=\int\limits_{\gamma_b}^{t} k_B(t,s)x(s)ds=\int\limits_{\gamma_b}^{\beta} I_{[\gamma_b,t]}(s)k_B(t,s)x(s)ds,
\end{align*}
and then applying Theorem \ref{thmBothIntOPKernelsGen}. Since
$[\gamma_a,\beta]\cap[\gamma_b,\beta]=[\gamma_b,\beta]$, the relation $AB=\delta BA$ holds for some $\delta\in\mathbb{R}\setminus\{0\}$ if and only if
\begin{enumerate}[label=\textup{\arabic*.}, ref=\arabic*, leftmargin=*]
  \item for almost every $(t,\tau)\in [\alpha,\beta]\times [\gamma_b,\beta]$,
   \begin{align*}
    &\resizebox{0.95\hsize}{!}{$\displaystyle \int\limits_{\gamma_a}^{\beta} I_{[\gamma_a,t]}(s)k_A(t,s)I_{[\gamma_b,s]}(\tau)k_B(s,\tau)ds=
           \delta\int\limits_{\gamma_b}^{\beta} I_{[\gamma_b,t]}(s)k_B(t,s)I_{[\gamma_a,s]}(\tau)k_A(s,\tau)ds
           $}
       \\
       & \Leftrightarrow  \int\limits_{\tau}^{t} k_A(t,s)k_B(s,\tau)ds=\delta\int\limits_{\tau}^{t} k_B(t,s)k_A(s,\tau)ds.
        \end{align*}
  \item for almost every $(t,\tau)\in [\alpha,\beta]\times [\gamma_a,\gamma_b]$,
        \begin{align*}
           & \delta\int\limits_{\gamma_b}^{\beta} I_{[\gamma_b,t]}(s)k_B(t,s)I_{[\gamma_a,s]}(\tau)k_A(s,\tau)ds=\delta\int\limits_{\gamma_b}^{t} k_B(t,s)k_A(s,\tau)ds=0
           \\
          & \Leftrightarrow \int\limits_{\gamma_b}^{t} k_B(t,s)k_A(s,\tau)ds=0.
        \end{align*}
\end{enumerate}
In the case of commutation relations $AB=\delta BA^n$, $n \geq 2$ applying this method might be tricky because of the products of indicator functions in $\mathbb{R}^{n+1}$ with $n+1$ variables.
\end{remark}

\begin{corollary}\label{CorCondCommutativityVolterraOperator}
Let $A:L_p(\mathbb{[\alpha,\beta]})\to L_p (\mathbb{[\alpha,\beta]})$, $B:L_p(\mathbb{[\alpha,\beta]})\to L_p(\mathbb{[\alpha,\beta]})$  be linear operators on
$L_p(\mathbb{[\alpha,\beta]})$, $1<p<\infty$, defined for almost every $t\in[\alpha,\beta]$ by
$$
  (Ax)(t)= \int\limits_{\alpha}^{t} k_A(t,s)x(s)ds,\quad (Bx)(t)= \int\limits_{\alpha}^{t} {k}_B(t,s)x(s)ds,
$$
where $\alpha,\gamma\in \mathbb{R}$, $\alpha\le t\le \beta$, and $k_A:[\alpha,\beta]^2\to\mathbb{R},\ {k}_B:[\alpha,\beta]^2\to\mathbb{R}$ are measurable functions. 
Let
$
  \Omega_{k_B} =\{(t,s,\tau)\in[\alpha,\beta]^3:\ {k}_B(t,s)\not=0 \ \mbox{and}\ {k}_B(s,\tau)\not=0\}.
$
Then, $AB=\delta BA$ for given $\delta\in\mathbb{R}\setminus \{0\}$ if the following conditions are fulfilled:
\begin{enumerate}[label=\textup{\arabic*.}, ref=\arabic*]
  \item  for almost every $(t,s,\tau)\in \Omega_{k_B}$ we have $k_A(t,s)=k_A(s,\tau)=0$  or $\delta=1$. And,
   if $\delta=1$ then for almost every $(t,s,\tau)\in \Omega_{k_B}$ we have $k_A(t,s)=\lambda k_B(t,s)$, $k_A(s,\tau)=\lambda k_{B}(s,\tau)$ for some real constant $\lambda$.
  \item for almost every $(t,s,\tau)\in [\alpha,\beta]^3\setminus\Omega_{k_B}$, the sets $\Omega_{g_A,g_B}$ and $\Omega_{h_A,h_B}$ defined by
  \begin{gather*}
   \Omega_{g_A,g_B}\stackrel{\rm def}{=} {\rm supp}\, g_A \cap {\rm supp }\, g_B  \cap   ([\alpha,\beta]^3\setminus \Omega_{k_B}),\\
   \Omega_{h_A,h_B}\stackrel{\rm def}{=} {\rm supp}\, h_A \cap {\rm supp }\, h_B  \cap   ([\alpha,\beta]^3\setminus \Omega_{k_B}),\\
  g_A:[\alpha,\beta]^3\to \mathbb{R},\ g_B:[\alpha,\beta]^3\to \mathbb{R},\ h_B:[\alpha,\beta]^3\to \mathbb{R},\ h_B:[\alpha,\beta]^3\to \mathbb{R},\\
      g_A(t,s,\tau)= k_A(t,s),\quad  g_B(t,s,\tau)= k_B(s,\tau),\\
      h_A(t,s,\tau)= k_A(s,\tau),\quad  h_B(t,s,\tau)= k_B(t,s).
    \end{gather*}
have measure zero in $\mathbb{R}^3$.
  \end{enumerate}
\end{corollary}

\begin{proof}
    By Theorem  \ref{ThmVolterraOpCommutativityCond}, when $\alpha=\gamma_a=\gamma_b$, we only remain with the condition \ref{volterraQuantumPlaneRelationCondition1} since the set $\Delta$ 
    has measure zero in $\mathbb{R}^2$ when $\gamma_a=\gamma_b$, so condition \ref{volterraQuantumPlaneRelationCondition2}
   can be removed. So, by direct computation we get a particular case of condition \ref{volterraQuantumPlaneRelationCondition1}, and thus $AB=\delta BA$. \qed
\end{proof}

  \begin{remark}
  It would be interesting to investigate necessary conditions for commutativity of Volterra operators. This involves studying properties of the set $\Omega_{k_B}$ and  the kernels $k_A(\cdot,\cdot)$, $k_B(\cdot,\cdot)$.
  \end{remark}

\begin{proposition}\label{propNecessaryCondCommuativityOpIntVoltTypeSimple}
Let $1< p<\infty$, $1<q<\infty$ be such that $\frac{1}{p}+\frac{1}{q}=1$, and the operators $A:\,L_p(\mathbb{[\alpha,\beta]})\to L_p(\mathbb{[\alpha,\beta]})$ and $B:\,L_p(\mathbb{[\alpha,\beta]})\to L_p(\mathbb{[\alpha,\beta]})$ be defined, for almost every $t$, by
\begin{gather*} 
  (Ax)(t)= \int\limits_{\alpha}^{t} a(t)c(s)x(s)ds,\quad (Bx)(t)= \int\limits_{\alpha}^{t} b(t)e(s)x(s)ds,
\end{gather*}
where $\alpha\in\mathbb{R}$, $a, b, c, e$ are measurable simple functions,  $a, b \in L_p(\mathbb{[\alpha,\beta]})$ and $c, e\in L_q(\mathbb{[\alpha,\beta]})$.
Then, if $AB=\delta BA\not=0$, for given $\delta\in\mathbb{R}\setminus \{0\}$, then
 ${\rm supp}\, \{(\delta-1)abce\}$ is a set of measure zero.
\end{proposition}

\begin{proof}
 Let $n\ge 1$ and we compute $AB$, $A^n$ and $BA^n$. Let $a$, $b$, $c$ and $e$ be measurable simple functions
 defined as follows
\begin{align*}
& a(t)=\sum_{k=1}^{\infty} a_k I_{D_{a_k}}(t), \quad b(t)=\sum_{i=1}^{\infty} b_i I_{D_{a_i}}(t),\\ & c(t)=\sum_{j=1}^{\infty} c_j I_{D_{c_j}}(t), \quad e(t)=\sum_{l=1}^{\infty} e_l I_{D_{e_l}}(t),
\end{align*}
where $\{D_{a_k}\}$, $\{D_{b_i}\}$, $\{ D_{c_j}\}$, $\{D_{e_l}\}$ are measurable partitions of $[\alpha,\beta]$,
$a_k$, $b_i$, $c_j$, $e_l$ are constants, $I_E$ is the indicator (characteristic) function of the set $E$.

Note that $ \bigcup\limits_{i,i,k,l=1}^{\infty}\{ D_{a_k} \cap D_{b_i} \cap D_{c_j}\cap D_{e_l} \}=\bigcup\limits_{m=1}^\infty P_m $ is a partition of $[\alpha,\beta]$. Therefore, for $x(\cdot)=1$  and $t\in P_1=D_{a_1} \cap D_{b_1} \cap D_{c_1}\cap D_{e_1}$ and by applying the dominated convergence theorem we have
\begin{align*}
& \hspace{0mm}(ABx)(t)=\hspace{0mm}\int\limits_{\alpha}^{t}  a_1c_1 I_{D_{a_1}}(t)
  I_{D_{c_1}}(s_1) 
\hspace{0mm}  b_1e_1 (s_1I_{D_{e_1}}(s_1)-\gamma_{e_1})I_{D_{b_1}}(s_1) ds_1\\
&=
a_1c_1b_1e_1
\big(\frac{t^2}{2} I_{D_{b_1}\cap D_{e_1}\cap D_{a_1}\cap D_{c_1} }(t)
-\frac{\gamma_{c_1,b_1,e_1}^2}{2}I_{D_{e_1}\cap D_{c_1}\cap D_{b_1}}(\gamma_{c_1,b_1,e_1})I_{D_{a_1}}(t)
\\
&
- t\cdot \gamma_{e_1}I_{D_{b_1}\cap D_{e_1}\cap D_{a_1}\cap D_{c_1} }(t)
 +\gamma_{e_l}
\gamma_{c_1,b_1,e_1}I_{D_{e_1}\cap D_{c_1}\cap D_{b_1}}(\gamma_{c_1,b_1,e_1})I_{D_{a_1}}(t)\big)
\\
&
= a_1c_1b_1e_1
\big(\frac{t^2}{2} I_{D_{b_1}\cap D_{e_1}\cap D_{a_1}\cap D_{c_1} }(t)
-\frac{\gamma_{c_1,b_1,e_1}^2}{2}I_{D_{e_1}\cap D_{c_1}\cap D_{b_1}}(\gamma_{c_1,b_1,e_1})I_{D_{a_1}}(t)
\\
&
- t\cdot \gamma_{e_1}I_{D_{b_1}\cap D_{e_1}\cap D_{a_1}\cap D_{c_1} }(t)
  +\gamma_{e_1}
\gamma_{c_1,b_1,e_1}I_{D_{e_1}\cap D_{c_1}\cap D_{b_1}}(\gamma_{c_1,b_1,e_1})I_{D_{a_1}}(t)\big),
\end{align*}
where $\gamma_{e_l}=\inf D_{e_l}$, $\,\gamma_{c_j,b_i, e_l}=\left\{\begin{array}{cc}
 \inf \{D_{c_j}\cap D_{b_i}\cap D_{e_l}\},\ & \mbox{ if } \mu(D_{c_j}\cap D_{b_i}\cap D_{e_l})>0\\
 \beta +1,         &   \mbox{otherwise }.
\end{array}  \right.$

Similarly,
\begin{align*}
& \hspace{0mm}(BAx)(t)=\hspace{0mm}\int\limits_{\alpha}^{t}  b_1e_1 I_{D_{b_1}}(t)
  I_{D_{e_1}}(s_1) 
\hspace{0mm}  a_1c_1 (s_1I_{D_{a_1}}(s_1)-\gamma_{c_1})I_{D_{a_1}}(s_1) ds_1\\
&=
a_1c_1b_1e_1
\big(\frac{t^2}{2} I_{D_{b_1}\cap D_{e_1}\cap D_{a_1}\cap D_{c_1} }(t)
-\frac{\gamma_{c_1,a_1,e_1}^2}{2}I_{D_{e_1}\cap D_{c_1}\cap D_{a_1}}(\gamma_{c_1,a_1,e_1})I_{D_{b_1}}(t)
\\
&
- t\cdot \gamma_{c_1}I_{D_{b_1}\cap D_{e_1}\cap D_{a_1}\cap D_{c_1} }(t)
 +\gamma_{c_1}
\gamma_{c_1,a_1,e_1}I_{D_{e_1}\cap D_{c_1}\cap D_{a_1}}(\gamma_{c_1,a_1,e_1})I_{D_{b_1}}(t)\big)
\\
&
= a_1c_1b_1e_1
\big(\frac{t^2}{2} I_{D_{b_1}\cap D_{e_1}\cap D_{a_1}\cap D_{c_1} }(t)
-\frac{\gamma_{c_1,a_1,e_1}^2}{2}I_{D_{e_1}\cap D_{c_1}\cap D_{a_1}}(\gamma_{c_1,a_1,e_1})I_{D_{b_1}}(t)
\\
&
- t\cdot \gamma_{c_1}I_{D_{b_1}\cap D_{e_1}\cap D_{a_1}\cap D_{c_1} }(t)
 +\gamma_{c_1}
\gamma_{c_1,a_1,e_1}I_{D_{e_1}\cap D_{c_1}\cap D_{a_1}}(\gamma_{c_1,a_1,e_1})I_{D_{b_1}}(t)\big),
\end{align*}
where $\gamma_{c_j}=\inf D_{c_j}$, $\,\gamma_{c_j,a_i, e_l}=\left\{\begin{array}{cc}
 \inf \{D_{c_j}\cap D_{a_k}\cap D_{e_l}\},\ & \mbox{ if } \mu(D_{c_j}\cap D_{a_k}\cap D_{e_l})>0\\
 \beta +1,         &   \mbox{otherwise }.
\end{array}  \right.$
Therefore,
\begin{equation*}
 a_1c_1b_1e_1I_{D_{b_1}\cap D_{e_1}\cap D_{a_1}\cap D_{c_1} }(t)=\delta  a_1c_1b_1e_1I_{D_{b_1}\cap D_{e_1}\cap D_{a_1}\cap D_{c_1} }(t),
\end{equation*}
for almost every $t$, which implies $(\delta-1) a_1c_1b_1e_1 I_{D_{b_1}\cap D_{e_1}\cap D_{a_1}\cap D_{c_1} }(t)=0$ for almost every $t$.
By applying the same procedure as in the proof of Proposition \ref{propOpIntVoltTypeSimple}, we get
\begin{equation*}
\sum_{k,i,j,l=1}^{\infty} a_kc_jb_ie_lI_{D_{b_i}\cap D_{e_l}\cap D_{a_k}\cap D_{c_j} }(t)=\delta \sum_{k,i,j,l=1}^{\infty} a_kc_jb_ie_lI_{D_{b_i}\cap D_{e_l}\cap D_{a_k}\cap D_{c_j} }(t),
\end{equation*}
for almost every $t$, which implies $(\delta-1)\sum\limits_{k,i,j,l=1}^{\infty} a_kc_jb_ie_l I_{D_{b_i}\cap D_{e_l}\cap D_{a_k}\cap D_{c_j} }(t)=0$ for almost every $t$. This is equivalent to the set supp$\{(\delta-1)abce\}$ to have measure zero.
\qed
\end{proof}

\begin{proposition}\label{PropCommutativityVolterraOpZerobothSides}
Let $1<p<\infty$, and the linear operators $A:L_p(\mathbb{[\alpha,\beta]})\to L_p (\mathbb{[\alpha,\beta]})$ and  $B:L_p(\mathbb{[\alpha,\beta]})\to L_p(\mathbb{[\alpha,\beta]})$ be defined, for almost every $t\in[\alpha,\beta]$, by
\begin{gather*}
  (Ax)(t)= \int\limits_{\alpha}^{t} k_A(t,s)x(s)ds,\quad (Bx)(t)= \int\limits_{\alpha}^{t} {k}_B(t,s)x(s)ds,
\end{gather*}
where $\alpha,\gamma\in \mathbb{R}$, $\alpha\le t\le \beta$, and $k_A:[\alpha,\beta]^2\to \mathbb{R}$ and ${k}_B:[\alpha,\beta]^2\to \mathbb{R}$ are measurable functions.
Then $AB= BA=0$ if and only if the sets
  \begin{align}
    & {\rm supp}\, g_A \cap {\rm supp }\, g_B, \label{CondiZeroDiversorsVolterraOperatorgAgB}\\
    &{\rm supp}\, h_A \cap {\rm supp }\, h_B \label{CondiZeroDiversorsVolterraOperatorhAhB}\\
    & g_A:[\alpha,\beta]^3\to \mathbb{R},\ g_B:[\alpha,\beta]^3\to \mathbb{R},  \nonumber\\
    & h_A:[\alpha,\beta]^3\to \mathbb{R},\ h_B:[\alpha,\beta]^3\to \mathbb{R} \nonumber \\
    & g_A(t,s,\tau)=k_A(t,s)\ g_B(t,s,\tau)=k_B(s,\tau), \nonumber \\
    & h_A(t,s,\tau)=k_A(s,\tau),\ g_B(t,s,\tau)=k_B(t,s) \nonumber
  \end{align}
  have measure zero in $\mathbb{R}^3$.
\end{proposition}

\begin{proof}
By Fubini Theorem, changing the order of integration, we get
  \begin{gather*}
     (ABx)(t)=\int\limits_{\alpha}^{t} k_A(t,s)\big(\int\limits_{\alpha}^{s}{k}_B(s,\tau)x(\tau)d\tau\big)ds
    =\int\limits_{\alpha}^{t}x(\tau) \big(\int\limits_{\tau}^{t}k_A(t,s){k}_B(s,\tau)ds\big)d\tau
  \end{gather*}
  for almost every $t\in [\alpha,\beta]$. Then, $ABx=0$ for all $x\in L_p[\alpha,\beta]$ if and only if
  \begin{gather*}
    \int\limits_{\tau}^{t}k_A(t,s){k}_B(s,\tau)ds=0
  \end{gather*}
  for almost every $(t,\tau)\in[\alpha,\beta]^2$. This is equivalent to  the set in \eqref{CondiZeroDiversorsVolterraOperatorgAgB} having
  measure zero in $\mathbb{R}^3$.
 Applying the same idea to the operator $BA$, we get that $BA=0$ if and only if the set in
  \eqref{CondiZeroDiversorsVolterraOperatorhAhB}
  has measure zero in $\mathbb{R}^3$.
\end{proof}

\begin{example}
Let $1<p<\infty$, and the linear operators $A:L_p[0,1]\to L_p[0,1]$ and $B:L_p[0,1]\to L_p[0,1]$  be defined, for almost for all $t$, by
\begin{align*}
& (Ax)(t)=\int\limits_{0}^{t} a(t)c(s)x(s)ds, \quad  (Bx)(t)=\int\limits_{0}^{t} b(t)e(s)x(s)ds,\\
& a(t)=I_{[0,1/4]}(t)(t^4+1)-I_{[1/2,3/4]}(t), \quad b(t)=I_{[1/2,1]}(t), \\
& c(t)=I_{[0,1/2]}(t), \quad e(t)=I_{[1/4,1/2]}(t)(t^2+1)+I_{[3/4,1]}(t),
   \end{align*}
 where $I_\Lambda$ denotes the indicator function of the set $\Lambda$. Since
$ \mbox{\rm supp\,} \{ a(t)b(s)c(s)e(\tau) \}$ has measure zero in $\mathbb{R}^3$ and the set
$\mbox{\rm supp\,} \{b(t)e(s)a(s)c(\tau)\} $ has measure zero in $\mathbb{R}^3$,
Proposition \ref{PropCommutativityVolterraOpZerobothSides} yields
 $AB=BA=0$.
\end{example}

\section*{Acknowledgments}
This work was supported by the Swedish International Development Cooperation Agency (Sida), bilateral capacity development program in Mathematics with Mozambique. Domingos Djinja is grateful to the Mathematics and Applied Mathematics research environment MAM, Division of Mathematics and Physics, School of Education, Culture and Communication, M\"alardalen University for excellent environment for research in Mathematics.


\end{document}